\documentclass[12pt,letterpaper,oneside]{amsart}
\usepackage[a4paper, left=3cm, right=3cm, bottom=3cm,top=3cm]{geometry}
\usepackage{mathscinet}
\usepackage{amsmath, amsfonts, amssymb, amsthm}
\usepackage{xcolor}
\definecolor{blue}{rgb}{0,0.5,0.5}
\definecolor{ocean}{rgb}{0.00,0.26,0.50}
\usepackage[colorlinks = true, linkcolor=blue, citecolor=ocean]{hyperref}
\newtheorem{theorem}{Theorem}[section]
\newtheorem{proposition}[theorem]{Proposition}
\newtheorem{definition}[theorem]{Definition}
\newtheorem{remark}[theorem]{Remark}

\newtheorem{example}[theorem]{Example}

\newtheorem{corollary}[theorem]{Corollary}
\newtheorem{lemma}[theorem]{Lemma}

\title[Locally quasiconvex hyperbolic TDLC groups and CT maps]{Combination of Locally Quasiconvex hyperbolic TDLC groups and Cannon-Thurston maps}
\author{Swarnali Datta, Arunava Mandal, and Ravi Tomar}

\address{Department of Mathematics,
Indian Institute of Technology Roorkee,
Uttarakhand 247667, India}
\email{swarnali\_d1@ma.iitr.ac.in}

\address{Department of Mathematics,
Indian Institute of Technology Roorkee,
Uttarakhand 247667, India}
\email{arunava@ma.iitr.ac.in}

\address{Beijing International Center for Mathematical Research, Peking University, No. 5 Yiheyuan Road Haidian District, Beijing, P.R.China 100871}
\email{ravitomar547@gmail.com}

\begin{document}
\maketitle

\begin{abstract}
In this article, we study acylindrical graphs of groups, local quasiconvexity, and Cannon-Thurston maps in the setting of totally disconnected locally compact (TDLC) hyperbolic groups, extending several fundamental notions and results from discrete hyperbolic groups to this broader context. Leveraging Dahmani's technique and a topological characterization of hyperbolic TDLC groups in terms of uniform convergence groups given by Carette-Dreesen, we prove a combination theorem for an acylindrical graph of hyperbolic TDLC groups and give an explicit construction of the Gromov boundary of the fundamental group of the given graph of groups. 
 Using the description of the Gromov boundary, we prove our main result: a combination theorem for an acylindrical graph of locally quasiconvex hyperbolic TDLC groups. Further, we generalise the work of Mosher, proving the existence of quasiisometric sections for a given short exact sequence of hyperbolic TDLC groups. This leads us to prove the existence of a Cannon-Thurston map for a normal hyperbolic subgroup of a hyperbolic TDLC group, generalising a theorem of Mj.
\end{abstract}

\subjclass{\it 2020 Mathematics Subject Classification:} Primary 20F65, 20F67; Secondary 22D05 

\keywords{\bf Keywords:} Hyperbolic TDLC groups, Convergence TDLC groups, Local Quasiconvexity, Cannon-Thurston maps.

\setcounter{tocdepth}{2}
\tableofcontents

\section{Introduction}\label{section:intro}

In his essay, Gromov \cite{gromov-hypgps} introduced quasiconvex subgroups while developing the theory of discrete hyperbolic groups. The discovery of these subgroups triggered extensive research on the subject (see, for instance, \cite{gitik, kapovich-local-qc, gitik-ping-pong, swenson, GMRS, hruska-wise}). 
%These subgroups play a role akin to convex subspaces of geodesic metric spaces and are the most natural from a geometric point of view. 
Quasiconvex subgroups enjoy various important properties; for example, a quasiconvex subgroup of a hyperbolic group is itself a hyperbolic group, and the intersection of two quasiconvex subgroups is again quasiconvex. It can also be characterized in terms of the behaviour of the boundaries, namely the existence of an injective Cannon-Thurston map (see \cite{mahan-on-a-theorem-of-scott}). In the study of intrinsic properties of hyperbolic groups, among various themes, there has been a wide range of studies revolving around {\em local quasiconvexity}, with fundamental contributions by Gitik \cite{gitik} and I. Kapovich \cite{kapovich-local-qc}.

In this article, we establish a foundational framework for the study of quasiconvex subgroups, local quasiconvexity, and Cannon-Thurston maps for hyperbolic totally disconnected locally compact (TDLC) groups. There has been growing interest in extending the coarse geometric framework of discrete hyperbolic groups to the setting of TDLC groups, which naturally arise as automorphism groups of locally finite graphs and buildings, algebraic groups over non-Archimedean local fields, and other non-discrete contexts, see the work of Kr\"{o}n-M\"{o}ller \cite{kron-moller}, Baumgartner-M\"{o}ller-Willis \cite{baumgartner-moller-willis-flatrank-one}, Carette-Dreesen \cite{mathieu-dennis-locally-compact-convergence}, Caprace-Cornulier-Monod-Tessera \cite{caprace-mood-amenable-hyp}, Arora-Castellano-Cook-Pedroza \cite{arora-castellano-pedroja}, see also \cite{arora-pedroja}, \cite{combination-theorem}, and references there in. Motivated by the work of Gitik \cite{gitik} and Kapovich \cite{kapovich-local-qc}, the main aim of this article is to prove a combination theorem for a finite graph of locally quasiconvex hyperbolic TDLC groups. However, in contrast to their proof, our approach is different, which uses Dahmani's boundary construction and the characterization of quasiconvexity in terms of the existence of an injective Cannon-Thurston map. To this end, we first establish several important properties for quasiconvex subgroups of a hyperbolic TDLC group, including the intersection of quasiconvex subgroups being quasiconvex, the limit set intersection property, and the fact that the height of a quasiconvex subgroup is finite. However, not all properties of quasiconvex subgroups are inherited from the discrete setup to the TDLC setup, see Remark \ref{remark:normalizer-centralizer}.
A first step towards our goal is the requirement of hyperbolicity of the fundamental group of the given graph of hyperbolic TDLC groups. Therefore, we introduced the notion of acylindrical graphs of groups in our setting, and by adapting the machinery on geometrically finite convergence groups developed by Dahmani \cite{dahmani-comb} in our context, we prove a combination theorem for a finite acylindrical graph of hyperbolic TDLC groups. Moreover, we give an explicit construction of the Gromov boundary of the fundamental group of the given graph of groups. This description, along with the characterization of quasiconvexity, allows us to fulfill our goal.
 As an independent interest, we prove that given a short exact sequence of open continuous homomorphisms of compactly generated hyperbolic TDLC groups
$1 \to H \to G \to Q \to 1,$ there exists a Cannon-Thurston map for the pair $(H, G)$.
    
\subsection{Acylindrical graphs of hyperbolic TDLC groups} The notion of acylindrical action originates from the work of Sela \cite{sela-acylindrical}. In the following definition, we adapt this notion in the TDLC setup. Let $(\mathcal G,\mathcal Z)$ be a graph of topological groups with Bass-Serre tree $T$ (see Subsection \ref{subsection:graphs-of-top-groups}). 
\begin{definition}\label{defn:acyl-action}
    For $k\geq 0$, the action of the fundamental group $G$ of $(\mathcal G,\mathcal Z)$ on $T$ is said to be $k$-{\em acylindrical} if the pointwise stabilizer of a geodesic of length greater than $k$ is compact. The action of $G$ on $T$ is said to be {\em acylindrical} if it is $k$-acylindrical for some $k\geq 0.$
\end{definition}
For infinite finitely generated groups, the notion of Cayley graph is a fundamental tool-kit in studying their geometry. In the same spirit, to study large-scale properties of TDLC groups, Kr\"{o}n and M\"{o}ller \cite{kron-moller} introduced the notion of a  Cayley-Abels graph (with the name rough Cayley graph but now commonly known as Cayley-Abels graph) for compactly generated TDLC groups, see Subsection \ref{subsection:2.2} for the definition and further details. Moreover, they showed that any two Cayley-Abels graphs of a compactly generated TDLC group are quasiisometric (Theorem \ref{theorem:existenc-qi-invariance-cayley-abel-graph}). This leads one to define the notion of a hyperbolic TDLC group: a compactly generated TDLC group $G$ is said to be {\em hyperbolic} if some (any) Cayley-Abels graph of $G$ is a Gromov hyperbolic space. Motivated by the celebrated combination theorem of Bestvina-Feighn \cite{BF}, several people proved combination theorems for hyperbolic and relatively hyperbolic groups. We list a few of them \cite{ilyakapovichcomb, dahmani-comb, mj-reeves, mahan-sardar, martin1, tomar, tomar-rel-acyl}. In \cite{ilyakapovichcomb}, I. Kapovich proved a combination theorem for a finite acylindrical graph of hyperbolic groups. Using a completely different technique, Dahmani \cite{dahmani-comb} proved an analogous theorem in the context of relatively hyperbolic groups. We adapt Dahmani's technique in the realm of acylindrical graphs of hyperbolic TDLC groups, and prove the following combination theorem. {\em Throughout the paper, on a finite graph of topological groups, we use the topology as in Proposition \ref{proposition-topology-graphs-of-groups}}.
\begin{theorem}\label{theorem:acyl-combination-theorem}
Let $(\mathcal G,\mathcal Z)$ be a finite graph of groups with fundamental group $G$ such that the following hold:
\begin{enumerate}
    \item The vertex groups are hyperbolic TDLC.
    \item The edge groups are quasiconvex in the adjacent vertex groups.
    \item The action of $G$ on the Bass-Serre tree of $(\mathcal G,\mathcal Z)$ is acylindrical.
\end{enumerate}
    Then, $G$ is a hyperbolic TDLC group, and the vertex groups are quasiconvex in $G$.
\end{theorem}
A subgroup $C$ in a topological group $A$ is said be be {\em weakly malnormal} if, for all $a\in A\setminus C$, $a^{-1}Ca\cap C$ is compact. If $G=A\ast_C B$ is an amalgamated free product of topological groups such that $C$ is weakly malnormal in $A$, then the action of $G$ on its Bass-Serre tree is $2$-acylindrical. Thus, as an application of Theorem \ref{theorem:acyl-combination-theorem}, if $A$ and $B$ are hyperbolic TDLC groups and $C$ is quasiconvex in $A$ and $B$ such that $C$ is weakly malnormal in $A$, then $G$ is hyperbolic. For a non-compact weakly malnormal subgroup of a hyperbolic TDLC group, see Example \ref{example:malnormal-example}.

Gitik et al. introduced the notion of the height of a subgroup of a group and proved that quasiconvex subgroups of hyperbolic groups have finite height \cite{GMRS}. Swarup asked the converse, and it remains an open question even when the subgroup is malnormal, that is, a subgroup of height $1$ \cite{bestvinahp}. Mj (formerly Mitra) proved that if a hyperbolic group $G$ splits as an amalgamated free product or HNN-extension of hyperbolic groups, then the edge group has finite height in $G$ if and only if it is quasiconvex in $G$ \cite{mitra-ht}. Recently, the third-named author \cite{tomar-rel-height} generalised it in the setting of relatively hyperbolic groups (see also \cite{pal-height-splitting}). As an application of Theorem \ref{theorem:acyl-combination-theorem}, in the TDLC setting, we prove the following result, generalising \cite[Theorem 4.6]{mitra-ht}. Our proof of Theorem \ref{theorem:height-splitting} gives a new proof even in the discrete setting.

\begin{theorem}\label{theorem:height-splitting}
Suppose $G$ is a hyperbolic TDLC group which splits as $A\ast_C B$ or $A\ast_C$, where $A, B$ are hyperbolic TDLC groups, and $C$ is an open quasiconvex subgroup of $A$ and $B$. Then, $C$ is quasiconvex in $G$ if and only if $C$ has finite height in $G$.
\end{theorem}
One is referred to Definition \ref{defn:height} for the height of a subgroup of a TDLC group.
%%%%%%%%%%%%%%%%%%%%%%%%%%%%%%%%%%%%%%%%%%%%%%%%%%%%%%%%
\subsection{Combination of locally quasiconvex hyperbolic TDLC groups}
A finitely generated group is said to be {\em locally quasiconvex} if every finitely generated subgroup is quasiconvex. An important feature of locally quasiconvex hyperbolic groups is that they satisfy the {\em Howson property} (the intersection of two finitely generated subgroups is again finitely generated). It remains an open question whether a hyperbolic group satisfying the Howson property is locally quasiconvex. For discrete hyperbolic groups, Gitik \cite{gitik} and Kapovich \cite{kapovich-local-qc} proved combination theorems for locally quasiconvex groups. Recently, the third-named author proved a combination theorem for certain complexes of locally quasiconvex hyperbolic groups \cite{tomar-remark-finite-edge-groups}. We say that a hyperbolic TDLC group $G$ is {\em locally quasiconvex} if every open compactly generated subgroup of $G$ is quasiconvex. Since, by Lemma \ref{lemma:properties-of-qc}(3), open quasiconvex subgroups of a hyperbolic TDLC group are hyperbolic, locally quasiconvex hyperbolic groups are coherent, i.e. open compactly generated subgroups are compactly presented, see the work of Arora-Pedroza \cite{arora-pedroja} for more on coherent topological groups. Using the description of the Gromov boundary given in the proof of Theorem \ref{theorem:acyl-combination-theorem}, under the weaker assumptions than the Howson property, we prove the following combination theorem for locally quasiconvex hyperbolic TDLC groups.

\begin{theorem}\label{theorem:combination-of-locally-qc} Let $(\mathcal G,\mathcal Z)$ be a finite graph of topological groups with fundamental group $G$ such that the following hold:
\begin{enumerate}
    \item The vertex groups are locally quasiconvex hyperbolic TDLC groups.
    \item The edge groups are quasiconvex in the adjacent vertex groups.
    \item The action of $G$ on the Bass-Serre tree of $(\mathcal G,\mathcal Z)$ is acylindrical.
    \item The intersection of a compactly generated subgroup of $G$ with the conjugates of the edge groups in $G$ is compactly generated.
\end{enumerate}
      Then, $G$ is a locally quasiconvex hyperbolic TDLC group. Conversely, if $(\mathcal G,\mathcal Z)$ satisfies (1), (2), and $G$ is locally quasiconvex, then the vertex groups of $(\mathcal G,\mathcal Z)$ are locally quasiconvex in $G$. 
\end{theorem}
 When the edge groups in Theorem \ref{theorem:combination-of-locally-qc} are compact, conditions (2) and (3) are automatically satisfied, and hence $G$ is locally quasiconvex if and only if the vertex groups are locally quasiconvex (see Corollary \ref{corollary:locally-qc-compact-edge-groups}). Moreover, we prove that TDLC groups quasiisometric to locally finite trees are locally quasiconvex hyperbolic (Corollary \ref{corollary:quasiisometric-to-tree-locally-qc}). This, in particular, implies that $\rm{SL}(2,\mathbb{Q}_p)$ is locally quasiconvex, see Example \ref{example-local-qc}. A result related to a combination of quasiconvex subgroups is also discussed in Section \ref{section:combination-of-hyperbolic-TDLC-groups} (Theorem \ref{theorem:gitik-qc-combination}).    
%%%%%%%%%%%%%%%%%%%%%%%%%%%%%%%%%%%%%%%%%%%%%%%%%%%%%%%%%%%%%%
\subsection{Cannon-Thurston maps for extensions of hyperbolic TDLC groups}
    We recall that the notion of Cannon-Thurston (CT) maps was inspired by the work of Cannon and Thurston (\cite{CTpub}) on hyperbolic $3$-manifolds. It was introduced in geometric group theory by Mahan Mitra (Mj) in \cite{mitra-trees} and \cite{mitra-ct}.
		
{\em Given a map $f:Y\to X$ between two (Gromov) hyperbolic spaces, the CT map for $f$ is a continuous extension of $f$ to the Gromov boundaries $\partial f: \partial Y\to \partial X$.}
		
We refer to Subsection \ref{subsection:CT maps} for a somewhat detailed discussion. In particular, if $H<G$ are hyperbolic groups, one asks if the inclusion $H\to G$ admits the CT map $\partial H\to \partial G$. In this case, we also say that the pair $(H,G)$ admits a CT map. This question of Mj appeared in \cite[Question 1.19]{bestvinahp} and was subsequently answered negatively by Baker and Riley \cite{baker-riley}. The CT map in geometric group theory became popular mostly through the work of Mj. We refer to \cite{mahan-icm} for a detailed history of the development of CT maps. In \cite{mitra-ct}, Mj proved that given a short exact sequence of finitely generated groups $1 \to H \to G \to Q \to 1$ with $G, H$ hyperbolic, there exists a Cannon-Thurston map for the pair $(H,G)$. We aim to prove an analog of this theorem in the TDLC setting. To prove his theorem, Mj used the technology of quasiisometric section (see Section \ref{section:existence-of-qisection}) for such a short exact sequence introduced by Mosher \cite{mosher-hypextns}. In fact, using a quasiisometric section, he constructed a quasiconvex subset of a Cayley graph of $G$, called {\em Ladder}, to be able to check the sufficient condition for the existence of the CT map (Lemma \ref{lemma:mitra-criterion}). We adapt the machinery developed by Mosher \cite{mosher-hypextns} and Mj \cite{mitra-ct} in the context of hyperbolic TDLC groups and prove the following. An example related to this theorem is discussed in Example \ref{example:ct-map-example}.

\begin{theorem}\label{main theorem}
     Given a short exact sequence of open continuous maps of compactly generated TDLC groups $1 \to H \to G \to Q \to 1$ such that $G$ and $H$ are hyperbolic, 
     and $H$ is non-elementary, there exists a Cannon-Thurston map for the pair $(H,G)$.
   \end{theorem}

The paper is organized as follows. In Section \ref{section:1}, we fix notations, recall key definitions, outline the construction of the Bass-Serre trees of graphs of topological groups, and provide brief introductions to Cayley-Abels graphs, convergence TDLC groups, and Cannon-Thurston maps. In the same section, we make several observations and prove some results that are important to us, like Lemma \ref{lemma:cosets of cpt set bounded}, Lemma \ref{lemma:conjugation-qi}, and Proposition \ref{prop:uniform-convergence-action}. Section \ref{section:quasiconvexity} establishes relevant properties of quasiconvex subgroups of hyperbolic TDLC groups--on their limit sets (Proposition \ref{prop:limit-set-intersection-property}), height (Proposition \ref{prop:qc-implies-finite-height}), and a characterization relating quasiconvexity and Cannon-Thurston maps (Proposition \ref{prop:CT-injective-iff-qc}). In Section \ref{section:construction-of-Gromov-boundary}, we prove Theorem \ref{theorem:acyl-combination-theorem}, a combination theorem for acylindrical graphs of hyperbolic TDLC groups with an explicit boundary construction, and apply it to characterize quasiconvex subgroups in the splitting groups in terms of height (Theorem \ref{theorem:height-splitting}). Section \ref{section:combination-of-hyperbolic-TDLC-groups} contains our main result on locally quasiconvex hyperbolic TDLC groups (Theorem \ref{theorem:combination-of-locally-qc}), related corollaries, and a combination theorem for subgroups generated by two quasiconvex subgroups of a hyperbolic TDLC group. As preparation for the Cannon-Thurston map result in Section \ref{section:4}, Section \ref{section:existence-of-qisection} introduces quasiisometric sections and proves their existence for short exact sequences of compactly generated hyperbolic TDLC groups, and Section 7 concludes by establishing the existence of a Cannon–Thurston map.
%%%%%%%%%%%%%%%%%%%%%%%%%%%%%%%%%%%%%%%%%%%%%%%%%%%%%%%%%%%%%%%%%%%%%%%%%%%%%%%%%%%%%%%%%%%%%%%%%%%%%%%%%%%%%%%%%%%%%%%%%%%%%%%%%%%%%%%%%%%%%%%%
\section{Background}\label{section:1}
In this section, we recall some basic definitions and results that we will use in the paper. Throughout the paper, we have the following standard assumptions:
\begin{enumerate}
    \item All topological groups are second countable Hausdorff.
    \item All actions of a group on a topological (metric) space are by homeomorphisms (isometries).
    \item All maps between topological groups are continuous.
\end{enumerate}
\begin{definition}
    The action of a topological group $G$ on a topological space $X$ is said to be {\em proper} if, for every compact set $K\subset X$, the set $\{g \in G: gK \cap K \neq \phi\}$ is relatively compact, i.e. its closure is compact.
\end{definition}

\subsection{Coarse geometry}
 In this subsection, we recall some standard notions of coarse geometry and set up some notation and conventions. Let $X$ be a metric space. For all $x,y\in X$, their distance in $X$ is denoted by $d_X(x,y)$ or
simply $d(x,y)$ when $X$ is understood. For $x\in X$, we denote by $B_r(x)$ the closed ball of radius $r$ with center $x$. 
For any $A\subset X$ and $D\geq 0$ we denote the closed $D$-neighborhood, i.e. 
$\{x\in X: d(x,a)\leq D \mbox{ for some }\, a\in A\}$ by $N_D(A)$. For $A, B\subset X$ we shall denote by $d_X(A,B)$ the
quantity $\inf\{d_X(a,b):a\in A, b\in B\}$. The {\em Hausdorff distance} between $A,B$ in $X$ is defined to be
$$d^{Haus}(A,B):= \inf\{D\geq 0: A\subseteq N_D(B), B\subseteq N_D(A)\}.$$

Suppose $x,y \in X$. A \emph{geodesic (segment)} joining $x$ to $y$ is an 
isometric embedding $\alpha: [a,b] \to X$ where $[a,b]\subset \mathbb R$ is an interval such that $\alpha(a)=x,\alpha(b)=y$. 
If any two points of $X$ can be joined by a geodesic segment
then $X$ is said to be a {\em geodesic metric space}. 
In this paper, graphs are assumed to be connected
and it is assumed that each edge is assigned a unit length so that the graphs are naturally geodesic metric spaces 
(see \cite[Section 1.9, I.1]{bridson-haefliger}). For a graph $\Gamma$, we denote the set of vertices and edges of $\Gamma$ by $V(\Gamma)$ and $E(\Gamma)$, respectively. The following are some other definitions that are important for us. Suppose that $X$ and $Y$ are metric spaces.

\begin{enumerate}
	
	\item	Let $\eta:[0,\infty)\to [0,\infty)$ be any map. A map $f:X\rightarrow Y$ is said to be a	$\eta$-\emph{proper embedding} if $d_Y(f(x),f(x'))\leq M$ implies $d_X(x,x')\leq \eta(M)$ for all $x,x'\in X$. A map $f:X\to Y$ is called a proper embedding if it is a $\eta$-proper embedding for some map $\eta:[0,\infty)\to [0,\infty)$.

	%\item If $L\geq 0$ then an {\em $L$-Lipschitz}	map $f:X\to Y$ between two metric spaces is one such that $d_Y(f(x),f(x'))\leq Ld_X(x,x')$ for all $x,x'\in X$. A $1$-Lipschitz map will simply be called a Lipschitz map. 
	
	%\item A map $f:X\to Y$ is said to be $L$-{\em coarsely Lipschitz} for a constant $L\geq 0$ if $d_Y(f(x), f(x'))\leq L+Ld_X(x,x')$ for all $x,x'\in X$. A map $f:X\to Y$ is said to be {\em coarsely Lipschitz} if it is $L$-coarsely Lipschitz for some $L\geq 0$.
	\item  Given $\lambda\geq 1,\epsilon\geq 0$, a map 
	$f:X\rightarrow Y$ is said to be a \emph{$(\lambda,\epsilon)$-quasiisometric (qi) embedding} if for all $x,x'\in X$ we have,
	$$\dfrac{1}{\lambda}d_X(x,x')-\epsilon\leq d_Y(f(x),f(x'))\leq \lambda d_X(x,x')+\epsilon.$$
	A $(\lambda,\lambda)$-qi embedding is simply called a $\lambda$-qi embedding.
	
	The map $f$ is said to be $(\lambda,\epsilon)$-\emph{quasiisometry} if $f$ is a $(\lambda,\epsilon)$-quasiisometric embedding
	and moreover, $N_D(f(X))=Y$ for some $D\geq 0$.  
	
	\item  A $(\lambda,\epsilon)$-\emph{quasigeodesic} in a metric space $X$ is a $(\lambda,\epsilon)$-quasiisometric embedding from an interval $I\subset \mathbb{R}$ in $X$. A (quasi)geodesic $\alpha:I\to X$ is called a {\em (quasi)geodesic ray} if $I=[0,\infty)$ and it is called a {\em (quasi)geodesic} 
	line if $I=\mathbb R$. 
	%A $(\lambda, \lambda)$-quasigeodesic segment or ray or line will simply be called a $\lambda$-quasigeodesic segment or ray or line, respectively.
	
\end{enumerate}
%%%%%%%%%%%%%%%%%%%%%%%%%%%%%%%%%%%%%%%%%%%%%%%%%%%%%%%%%%%%%%%%%%%%%%%%%%%%%%%%%%%%%%%%%%%%%%%%%%%%%%%%%%%%%%%%%%%%%%%%%%%%%%%%%%%%%%%%%%
\subsection{Graphs of topological groups}\label{subsection:graphs-of-top-groups}
In geometric group theory, the notion of graphs of groups is classical. One is referred to Serre's book on trees \cite{serre-trees} for a detailed account of definitions and results.
\begin{definition}\label{definition-graphs-of-groups}
Let $\mathcal Z$ be a graph. A {\em graph of topological groups} $(\mathcal G,\mathcal Z)$ over $\mathcal Z$ consists of the following data:
\begin{enumerate}
    \item[{$(1)$}] For each $v\in V(\mathcal Z)$, there is a topological group $G_v$ called the {\em vertex group}.
    \item[{$(2)$}] For each $e\in E(\mathcal Z)$, there is a topological group $G_e$ called the {\em edge group.}
    \item[{$(3)$}] For each edge $e\in E(\mathcal Z)$ with vertices $v$ and $w$, there are open topological embeddings $G_e\to G_v$ and $G_e\to G_w.$
\end{enumerate}
\end{definition}
We denote the fundamental group of the graph of groups $(\mathcal G,\mathcal Z)$ by $\pi_1(\mathcal G,\mathcal Z)$.
The next proposition shows that there is a canonical topology on $\pi_1(\mathcal G,\mathcal Z).$
\begin{proposition}\textup{\cite[Proposition 8.B.9 and Proposition 8.B.10]{cornulier-harpe-book-metric-geometry}}\label{proposition-topology-graphs-of-groups}
There is a unique topology on $\pi_1(\mathcal G,\mathcal Z)$ such that the inclusion map $G_v\to \pi_1(\mathcal G,\mathcal Z)$ is an open topological embedding. Moreover, we have the following:
\begin{enumerate}
\item[{$(1)$}] If the edge groups of $(\mathcal G,\mathcal Z)$ are locally compact, then $\pi_1(\mathcal G,\mathcal Z)$ is locally compact.

\item[{$(2)$}] If the vertex groups of $(\mathcal G,\mathcal Z)$ are compactly generated, then $\pi_1(\mathcal G,\mathcal Z)$ is compactly generated.
\end{enumerate}
\end{proposition}

In particular, if $\mathcal Z$ is an interval, $\pi_1(\mathcal G,\mathcal Z)$ is the amalgamated free product. If $\mathcal Z$ is a simple loop, then $\pi_1(\mathcal G,\mathcal Z)$ is an HNN extension. We now explain in detail these constructions.

Suppose $A,B$, and $C$ are topological groups such that $i_A:C\to A$ and $i_B:C\to B$ are topological isomorphisms onto open subgroups of $A$ and $B$, respectively. Then, we have an {\em amalgamated free product} $A\ast_C B$ given by the presentation $\langle S_A,S_B|R_A,R_B, i_A(c)=i_B(c) \text{ for all } c\in C \rangle$ if $\langle S_A|R_A\rangle$ and  $\langle S_B|R_B\rangle$ are presentations of $A$ and $B$, respectively. {\em Throughout the paper, we assume that the amalgamated free products are non-trivial, i.e. the edge group is not equal to both vertex groups}. If $C$ is a subgroup of $A$ and $\alpha$ is a topological isomorphism from $C$ to an open subgroup of $A$, then we have {\em HNN extension} $A\ast_C$ whose presentation is given by $\langle S_A,t|R_A, t^{-1}ct=\alpha(c) \text{ for all } c\in C \rangle$ if  $\langle S_A|R_A\rangle$ is a presentation of $A$. Let $G$ denote an amalgamated free product $A\ast_C B$ or an HNN extension $A\ast_C$. Then, by Proposition \ref{proposition-topology-graphs-of-groups}, there is a unique topology on $G$ such that the inclusions $A\to G, B\to G$, and $C\to G$ are topological isomorphisms onto open subgroups of $G$. {\em Throughout the paper, we use this topology on amalgamated free products and HNN extensions of topological groups.} We record the following lemma, which is relevant to us.

\begin{lemma}\cite[Lemma 2.13]{combination-theorem}\label{lemma:edge TDLC implies G is TDLC}
    If the edge groups of $(\mathcal G,\mathcal Z)$ are TDLC, then $\pi_1(\mathcal G,\mathcal Z)$ is TDLC.
\end{lemma}

\begin{remark}\label{remark:sufficient-to-prove-amalgam-hnn}
    In the forthcoming section, we prove combination theorems for hyperbolic TDLC groups and for locally quasiconvex hyperbolic TDLC groups. To prove such theorems, by induction, it is sufficient to prove the theorem for amalgamated free products and HNN extensions of topological groups.
\end{remark}

{\bf Bass-Serre trees of amalgamated free product and HNN extension:} Let $G = A \ast_{C} B$. The construction of the Bass--Serre tree of $G$ is classical \cite{serre-trees}. For the sake of completeness, we explain the construction as follows.

Let $\tau$ be a unit interval with vertices $v_{A}$ and $v_{B}$. Define an equivalence relation $\sim$ on $G \times \tau$ induced by the equivalences.
 $$(g_{1},v_{A}) \sim (g_{2},v_{A}) \text{ if } g_{1}^{-1} g_{2} \in A,$$
 $$(g_{1},v_{B}) \sim (g_{2},v_{B}) \text{ if } g_{1}^{-1} g_{2} \in B,$$
 $$(g_{1},t) \sim (g_{2},t) \text{ if } g_{1}^{-1} g_{2} \in C, t \in \tau.$$
 Then $G \times \tau /_\sim$ is called the {\em Bass-Serre tree} $T$ of $G$. 

 If $G=A\ast_C$ is an HNN extension, then the Bass-Serre tree $T$ of $G$ is defined as follows. The vertex set of $T$ is the set of left cosets of $A$ in $G$. The set of edges of $T$ is the set of left cosets of $C$ in $G$. For all $g \in G$, vertices $gA$ and $gtA$ are connected by the edge $gC$.
 
\vspace{.2cm} 
{\bf Notation:} We denote by $T$ the Bass-Serre tree of an amalgamated free product or an HNN extension. 
For $v\in V(T)$, we denote by $G_v$ the $G$-stabilizer of $v$. Similarly, for $e\in E(T)$, $G_e$ denotes the $G$-stabilizer of $e$.
%%%%%%%%%%%%%%%%%%%%%%%%%%%%%%%%%%%%%%%%%%%%%%%%%%%%%%%%%%%%%%%%%%%%%%%%%%%%%%%%%%%%%%%%%%%%%%%%%%%%%%%%%%%%%%%%%%%%%%%%%%%
\subsection{Basic properties of Cayley-Abels graphs}\label{subsection:2.2}
A finitely generated group can be realised as a geometric object by constructing its Cayley graph. In case of TDLC groups, an analogous construction was first given by Abels \cite{abels}, which is known as {\em Cayley-Abels graph}. Later in 2008, a less technical approach was introduced by Kr\"{o}n and M\"{o}ller \cite{kron-moller}, called the {\em rough Cayley graph}.
\begin{definition}
    A locally finite connected graph $X$ is said to be a {\em Cayley-Abels graph} of a TDLC group $G$ if $G$ acts vertex transitively on $X$ with compact open vertex stabilizers.
\end{definition} 

\begin{theorem}\label{theorem:existenc-qi-invariance-cayley-abel-graph}
For a TDLC group $G$, we have the following:
\begin{enumerate}
    \item \cite[Theorem 2.2]{kron-moller} There exists a Cayley-Abels graph of $G$ if and only if $G$ is compactly generated.
    \item \cite[Theorem 2.7]{kron-moller} Any two Cayley-Abels graphs of $G$ are quasiisometric.
\end{enumerate}
\end{theorem}

\vspace{.2cm}

 {\bf Construction of a Cayley-Abels graph:} Let $G$ be a compactly generated TDLC group and $U$ be a compact open subgroup of $G$. If $K$ is a compact generating set for $G$, then there is a finite symmetric set $A$ containing identity such that $K \subset AU$, $UAU=AU$ and $G=\langle A\rangle U$. Define a graph $\Gamma_{G}(K,U,A)$ whose vertex set is $\{gU: g \in G\}$ and edge set is $\{\{gU,gaU\}:g \in G, a \in A\}$. Then it is easy to see that $\Gamma_{G}(K,U,A)$ is a Cayley-Abels graph of $G$ \cite{kron-moller}. We write $\Gamma_{G}$ in place of $\Gamma_{G}(K,U,A)$ when $K,U,A$ are clear from the context.\label{construction}

 \begin{remark}\label{rem:canonical-cayley-abel-graph}
     If $X$ is a Cayley-Abels graph of a TDLC group $G$ with a compact generating set $K$, then there is a compact open subgroup $U$ and a finite set $A$ in $G$ such that $X= \Gamma_{G}(K,U,A)$ \cite[Section 2]{kron-moller}.
 \end{remark}

We now collect some properties of $\Gamma_G$ and the action of $G$ on $\Gamma_G$. For $gU,g'U\in V(\Gamma_G)$, define $d_A(gU,g'U):=\min\{n: g^{-1}g'U= a_{1}a_{2} \cdots a_{n}U, a_{i} \in A \text{ for } 1\leq i\leq n\}$. It is easy to check that $d_A$ is a metric on $V(\Gamma_G).$ 

 \begin{lemma}\label{lemma:induced-path-metric-on-vertices}
     The restriction of the path metric on $\Gamma_G$ to $V(\Gamma_{G})$ is the same as the metric $d_{A}$.
 \end{lemma}
 \begin{proof}
     Let us denote the induced path metric in $\Gamma_G$ by $\rho$. Let $gU,g'U\in V(\Gamma_G)$ and $\rho (gU,g'U)= n$. Then, there is a geodesic of length $n$ joining $gU$ and $g'U$ with vertices $gU, ga_{1}U, ga_{1}a_{2}U, \cdots  ,ga_{1}a_{2} \cdots a_{n}U= g'U $ for some $a_{1}, \cdots, a_{n} \in A$. This implies that $d_{A}(gU,g'U) \leq n$. Let if possible $d_A(gU,g'U)=m<n$. Then, $g^{-1}g'U= a'_{1}a'_{2} \cdots a'_{m}U$ for some $a'_{1}, \cdots, a'_{m} \in A$. Then, there is a path of length $m$ between $gU$ and $g'U$ and we get a contradiction. Hence, $\rho(gU,g'U)= d_{A}(gU,g'U)$ for every $g,g' \in G$.
 \end{proof}
 \begin{remark}
     We call $d_A$ a {\em word metric} on $\Gamma_G$. From now on we use $d_A$ as the metric on $V(\Gamma_G).$
 \end{remark}
  \begin{lemma}\label{lemma:cosets of cpt set bounded}
     Given a compact subset $K\subset G$ there exists a constant $c=c(K)\geq 0$ such that $d_A(U,KU):=\inf\{d_A(U,kU): k \in K\}\leq c$. Moreover, for any $k\in K$, $d_A(U,kU)\leq c.$
 \end{lemma}
 \begin{proof}
     Since $U$ is open in $G$, $\{kU:k \in K\}$ is an open cover of $K$. Then, for some $n\in\mathbb N$, there exists $k_{1},k_{2}, \cdots,k_{n} \in K$ such that $K \subset \bigcup_{i=1}^{n} k_{i}U$. This implies that $KU= \bigcup_{i=1}^{n} k_{i}U$. Therefore, for any $k \in K$, $kU=k_{j}U$ for some $1\leq j\leq n$. Hence,  $d_A(U,KU) = \min\{d_A(U,k_{i}U): i=1,2, \cdots, n\}$. Taking $c=\max\{d_A(U,k_{i}U): i=1,2, \cdots, n\}$, we are done.
 \end{proof}
 
 Suppose $H<G$ are compactly generated TDLC groups such that $H$ is open. Let $K_H$ and $K_G$ be compact generating sets of $H$ and $G$, respectively. Let $U$ be a compact open subgroup of $G$. Then, $U_H=U\cap H$ is compact and open in $H$. Let $A_G$ and $A_H$ be finite symmetric subsets of $H$ and $G$ respectively such that $K_H\subset A_HU_H$ and $K_G\subset A_GU.$  Let $\Gamma_H=\Gamma_H(K_H,U_H,A_H)$ and $\Gamma_G=\Gamma_G(K_G,U,A_G)$ be Cayley-Abels graphs of $H$ and $G$, respectively. We denote by $d_{H}$ and $d_{G}$ the metric on $\Gamma_H$ and $\Gamma_G$, respectively. We now assume that $A_H\subset A_G$. Note that distinct cosets of $U_H$ in $H$ give distinct cosets of $U$ in $H$. Thus, we have an embedding $\iota:\Gamma_H\to \Gamma_G$ that sends the vertex $hU_H$ to the vertex $hU$. Thus, we can treat $\Gamma_H$ as a subgraph of $\Gamma_G$. The proof of the following lemma follows similarly to the proof of Lemma \ref{lemma:induced-path-metric-on-vertices}.
 \begin{lemma}
     The restriction of the induced path metric on $\Gamma_{H}$ from $\Gamma_{G}$ to $V(\Gamma_H)$ is the same as the metric $d_{A_H}$ given by $d_{A_H}(hU_{H},h'U_{H}):=\min\{n: h^{-1}h'U_{H}= a_{1}a_{2} \cdots a_{n}U_{H}, a_{i} \in A_{H}\}$.                      \qed
 \end{lemma}

\begin{lemma}\cite[Lemma 2.3]{combination-theorem}\label{lemma:proper-embedding}
     The map $\iota: \Gamma_{H} \rightarrow \Gamma_{G}$ is a proper embedding.
 \end{lemma}

 \begin{remark}
     Throughout the paper, unless stated otherwise, when $H$ is an open compactly generated subgroup of a compactly generated TDLC group, $\Gamma_H$ and $\Gamma_G$ denote the Cayley-Abels graphs of $H$ and $G$, respectively, and $\iota:\Gamma_H\to\Gamma_G$ denotes the embedding.
 \end{remark}
%%%%%%%%%%%%%%%%%%%%%%%%%%%%%%%%%%%%%%%%%%%%%%%%%%%%%%%%%%%%%%%%%%%%%%%%%%%%%%%%%%%%%%%%%%%%%%%%%%%%%%%%%%%%%%%%%%%%%%

 \subsection{Hyperbolic TDLC groups}
 In \cite{gromov-hypgps}, Gromov introduced the notion of hyperbolic geodesic metric spaces and proved that hyperbolicity is an invariant of quasiisometry. Then, by Theorem \ref{theorem:existenc-qi-invariance-cayley-abel-graph}, we have the following well-defined notion of a hyperbolic TDLC group.
 \begin{definition}
     A compactly generated TDLC group $G$ is said to be {\em hyperbolic} if some (any) Cayley-Abels graph of $G$ is a Gromov hyperbolic space. 
 \end{definition}
 The following theorem shows that, in a hyperbolic geodesic metric space, geodesics and quasigeodesics having the same endpoints are not far apart from each other. 
 
 \begin{theorem}[Stability of quasigeodesics]\cite[Theroem 1.7, III.H]{bridson-haefliger}\label{stability of quasigeodesics}
     Given $\delta\geq 0,\lambda\geq 1,\epsilon\geq 0$ there exists $D=D(\delta,\lambda,\epsilon)$ such that the following hold:
     
     If $X$ is a $\delta$-hyperbolic space, $\alpha$ is a $(\lambda, \epsilon)$-quasigeodesic in $X$ and $\beta$ is a geodesic joining the endpoints of $\alpha$ in $X$, then the Hausdorff distance between $\alpha$ and $\beta$ is at most $D$.
 \end{theorem}
 It is easy to see that a finite extension of a discrete hyperbolic group is hyperbolic. An analogous phenomenon also holds in the TDLC setting. 
\begin{remark}[Compact extension of hyperbolic TDLC groups]\label{remark:compact-extension}
    Let $G$ be a compactly generated TDLC group and $N$ be a compact open normal subgroup of $G$. Then $G$ has a Cayley-Abels graph $\Gamma_{G}= \Gamma_{G}(K,N,A)$, where $K$ is a compact generating set of $G$ and $A$ is a finite symmetric set such that $NAN=AN$ and $G=\langle A \rangle N$. Hence, $G/N=\langle \{a_{i}N: a_{i} \in A\}\rangle$,i.e., $G/N$ is finitely generated. Therefore, $G/N$ has a Cayley graph which is $\Gamma_{G}$ itself by construction. Therefore, $G/N$ is quasiisometric to $G$. Thus, if $G/N$ is hyperbolic then $G$ is also hyperbolic.
\end{remark}

The following result is an analog of the Svarc-Milnor lemma for locally compact groups:
\begin{theorem}\cite[Theorem 4.C.5]{cornulier-harpe-book-metric-geometry}\label{milnor-svarc-lemma}
    Let $G$ be a locally compact group acting on a (coarsely connected) geodesic metric space $X$ properly and cocompactly. Then $G$ is compactly generated and for any fixed $x \in X$, the orbit map $g \mapsto gx: G \rightarrow X$ is a quasiisometry.
\end{theorem}

Let $H$ be a normal hyperbolic TDLC subgroup of a hyperbolic TDLC group. For each $g \in G$, there is an automorphism of $H$ taking $h$ to $g^{-1}hg$. This induces a map $\phi_{g}:V(\Gamma_H)\to V(\Gamma_H)$ defined by $\phi_{g}(hU_H) :=g^{-1}hgU_{H}$. This gives rise to a map from $\Gamma_{H}$ to itself, sending an edge joining $aU_H$ and $bU_H$ to a geodesic joining $\phi_{g}(aU_H)$ to $\phi_{g}(bU_H)$. We still call this map $\phi_{g}$ by abuse of notation. The following lemma is important to us and will be used frequently in the later sections.
\begin{lemma}\label{lemma:conjugation-qi}
 For each $g\in G$, $\phi_{g}$ is a quasiisometry.
 \end{lemma}
 \begin{proof}
      Since $V(\Gamma_H)$ is coarsely dense in $\Gamma_H$, it is sufficient to check that $\phi_g$ is a quasiisometry on $V(\Gamma_H)$. Note that any element $hU_H$ can be written as $hU_H=a_{1}a_{2}\cdots a_{n}U_H$, where $a_i \in A_H$ and $n$ is minimum. Now, since $A_H$ is finite, let $L:= \max\{d_{A_{H}}(U_H,\phi_{g}(aU_H)):a \in A_H\}$.

    Therefore, 
         \begin{align*}
             d_{A_{H}}(U_H,\phi_{g}(hU_H)) &= d_{A_{H}}(U_H,g^{-1}a_1 \cdots a_{n}gU_H)\\ &\leq d_{A_{H}}(U_H,g^{-1}a_{1}gU_H) + d_{A_{H}}(g^{-1}a_{1}gU_H,g^{-1}a_{1}a_{2}gU_H)+ \cdots \\ &+ d_{A_{H}}(g^{-1}a_1 \cdots a_{n-1}gU_H,g^{-1}a_1 \cdots a_{n}gU_H)\\ &=d_{A_{H}}(U_H,g^{-1}a_{1}gU_H) + %d_{A_{H}}(U_H,g^{-1}a_{2}gU_H)
            \cdots + d_{A_{H}}(U_H,g^{-1}a_{n}gU_H)\\ &\leq nL= Ld_{A_{H}}(U_H,hU_H).
     \end{align*}

Similarly, for $\phi_{g^{-1}}$, we get a constant $L'$ with $d_{A_{H}}(U_H,\phi_{g^{-1}}(hU_H)) \leq L'd_{A_{H}}(U_H,hU_H)$ for all $h \in H$. Therefore, $d_{A_{H}}(U_H,hU_H)= d_{A_{H}}(U_H,\phi_{g^{-1}}(\phi_{g}(hU_H))) \leq L' d_{A_{H}}(U_H,\phi_{g}(hU_H))$. Taking $\lambda=max\{L,L'\}$, we get $\frac{1}{\lambda}d_{A_{H}}(U_H,hU_H) \leq d_{A_{H}}(U_H,\phi_{g}(hU_H)) \leq \lambda d_{A_{H}}(U_H,hU_H)$. Therefore, $\phi_{g}$ is a quasiisometric embedding. Since $\phi_{g}$ is coarsely surjective $\phi_{g}$ is a quasiisometry.
 \end{proof} 

 {\bf Gromov boundary of a hyperbolic TDLC group:}
 In \cite{gromov-hypgps}, Gromov introduced the notion of a boundary of a hyperbolic geodesic metric space, which is called the Gromov boundary. He also proved that quasiisometric proper hyperbolic spaces have homeomorphic Gromov boundaries. This motivates one to define the {\em Gromov boundary} of a hyperbolic TDLC group $G$ as the Gromov boundary of some (any) Cayley-Abels graph $X$ of $G$. It is a set of the equivalence classes of geodesic rays in $X$ starting from a base point $x_0$, i.e.
\[
\partial G := \partial X= \{\, [\gamma] \mid \gamma\text{ is a geodesic ray in } X \text{ such that } \gamma(0)=x_0 \},
\]
where two rays are equivalent if their Hausdorff distance is finite. Define $\overline{X}:=X\cup\partial X.$ Then, with a natural topology, $\overline{X}$ is a compact metrizable space and $X$ is an open dense subset of $\overline{X}$, see \cite{bridson-haefliger} for more details.
A hyperbolic TDLC group $G$ is said to be {\em non-elementary} if the cardinality of the Gromov boundary of $G$ is at least $3$.
%%%%%%%%%%%%%%%%%%%%%%%%%%%%%%%%%%%%%%%%%%%%%%%%%%%%%%%%%%%%%%%%%%%%%%%%%%%%%%%%%%%%%%%%%%%%%%%%%%%%%%%%%%%%%%%%%%%%%%%%%%%%%%%%%
\subsection{Convergence TDLC groups}
 
The notion of a convergence group was introduced by Gehring and Martin \cite{gehring-martin} in
order to describe the dynamical properties of Kleinian groups acting on the ideal
sphere of (real) hyperbolic space. Later, this notion was generalised for groups acting
on compact Hausdorff spaces by several people such as Tukia, Freden, and Bowditch \cite{tukia},\cite{freden},\cite{bowditch-cgnce}. 

For a Hausdorff space $X$, let $X^3$ be the space of ordered triples with the product topology. We denote by $\theta(X)$ the set of distinct unordered triples of $X$ with the subspace topology from $X^3$. A discrete group $G$ acting on a compact Hausdorff space $X$ is said to be a {\em convergence group} if the induced action of $G$ on $\theta(X)$ is properly discontinuous. Moreover, $G$ is said to be a {\em uniform convergence group} if the action of $G$ on $\theta(X)$ is cocompact. In \cite{bowditch-jams}, Bowditch gave a topological characterisation of discrete hyperbolic groups. In \cite[Theorem A]{mathieu-dennis-locally-compact-convergence}, the authors extended it in the setting of locally compact hyperbolic groups. The goal of this subsection is to recall this result and other notions related to topological convergence groups that are relevant to us. We start with the following definition.

\begin{definition}
  A locally compact topological group $G$ acting on a compact Hausdorff space $X$ is said to be a {\em convergence group} if the induced action of $G$ on $\theta(X)$ is proper. Moreover, $G$ is said to be a {\em uniform convergence group} if the action of $G$ on $\theta(X)$ is cocompact.
\end{definition}
It is well-known to experts that a TDLC hyperbolic group acts as a uniform convergence group on its Gromov boundary (Proposition \ref{prop:uniform-convergence-action}). For the sake of completeness, we include a proof of this fact. The converse of this fact is the content of \cite{mathieu-dennis-locally-compact-convergence}. A continuous map between two topological spaces is said to be {\em proper} if the inverse image of a compact set is compact.

\begin{lemma}\label{proper surjective}
     Let $G$ be a locally compact topological group acting on locally compact Hausdorff spaces $X$ and $Y$. Suppose $f: Y \to X$ is a proper surjective $G$-equivariant map. Then, we have the following:
     \begin{enumerate}
         \item $G$ acts properly on $X$ if and only if it acts properly on $Y$ . 
    \item $G$ acts cocompactly on $X$ if and only if it acts cocompactly on $Y$.
     \end{enumerate}
 \end{lemma}

\begin{proof}
    (1) Suppose $G$ acts properly on $Y$. Let $K$ be a compact set in $X$. Since $f$ is proper, $f^{-1}(K)$ is compact in $Y$. Suppose $gK \cap K \neq \phi$ for some $g \in G$. Thus, there exists $x \in K$ such that $gx \in K$. Since $f$ is surjective, there exists $y \in Y$ such that $f(y)=x$. By $G$-equivariance of $f$, we have that $f(gy)=gf(y)=gx \in K$. This implies that $y,gy \in f^{-1}(K)$ and therfore $gf^{-1}(K) \cap f^{-1}(K) \neq \phi$. Thus, $N:=\{g \in G: gK \cap K \neq \phi\} \subseteq M:=\{g \in G: gf^{-1}(K) \cap f^{-1}(K) \neq \phi\}$. Since $M$ is relatively compact in $G$, $N$ is relatively compact in $G$. Hence, $G$ acts properly on $X$. Conversely, suppose $G$ acts properly on $X$. Let $L$ be a compact set in $Y$. Thus, $f(L)$ is compact in $X$. Suppose $gL \cap L \neq \phi$ for some $g \in G$. Since, $gf(L)=f(gL)$ and $f(gL) \cap f(L) \neq \phi$, $gf(L) \cap f(L) \neq \phi$. This implies that $N':=\{g \in G: gL \cap L \neq \phi\} \subseteq M':=\{g \in G: gf(L) \cap f(L) \neq \phi\}$. Since $M'$ is relatively compact in $G$, $N'$ is relatively compact in $G$. Hence, $G$ acts properly on $Y$.

    (2) Suppose $G$ acts cocompactly on $X$. Then, there exists a compact set $K\subseteq X$ such that $GK=X$. 
    Let $y \in Y$. 
    %Therefore, $f(y) \in X$.
    Then, there exist $g \in G$ and $k \in K$ such that $f(y)=gk$. This implies that $k=g^{-1}f(y) =f(g^{-1}y).$ This further implies that $g^{-1}y \in f^{-1}(K)$ and thus $y \in gf^{-1}(K)$. Therefore, $Y= Gf^{-1}(K)$, and hence, $G$ acts cocompactly on $Y$. Conversely, suppose $G$ acts cocompactly on $Y$. Then, there exists a compact set $K\subseteq Y$ such that $GK=Y$. 
    %Let $x \in X$. Then there is $y \in Y$ such that $f(y)=x$. Now, $y=g.k$ for some $g \in G, k \in K$. Then, $f(y)=f(g.k) \implies x=g.f(k) \implies x \in G.f(K)$. 
    By doing a similar computation as above, we see that $Gf(K)=X$. Hence, $G$ acts cocompactly on $X$.
\end{proof}

The idea of the proof of the following proposition is borrowed from \cite{bowditch-cgnce}.

\begin{lemma}\label{uc action non-discrete}
    Suppose a topological group $G$ acts properly, cocompactly, and vertex transitively on a locally finite hyperbolic graph $X$. Then, $G$ acts as a uniform convergence group on the Gromov boundary $\partial X$ of $X$.
\end{lemma}

\begin{proof}
    Let $X$ be $\delta$-hyperbolic and $\theta(\partial X)$ be the set of distinct unordered triples of $\partial X$. Let $Y \subseteq X \times \theta(\partial X)$ be the subset of pairs, $(a,\{x_{1}, x_{2}, x_{3}\})$ such that there exist bi-infinite geodesics, $\alpha_{1}, \alpha_{2}, \alpha_{3}$, with $\alpha_{i}$ connecting $x_{i}$ to $x_{i+1}$, with subscripts mod $3$, such that $d(a,\alpha_{i}) \leq k$ for some $k=k(\delta)$ and  $i \in \{1,2,3\}$ (see \cite[Exercise 11.86 and Corollary 11.89]{kapovich-drutu-book} for the existence of such $a$). We equip $X\times\theta(\partial X)$ with the product topology. Then, we see that $Y$ is a closed subset of $X \times \theta(\partial X)$. Moreover, since $G$ acts transitively on the vertices of $X$, up to increasing $k$ by $1$, the natural projections of $Y$ to $X$ and to $\theta(\partial X)$ are proper, surjective, and $G$-equivariant. Since $G$ acts properly and cocompactly on $X$, by Lemma \ref{proper surjective}, it follows that $G$ acts properly and cocompactly on $Y$. Again, using Lemma \ref{proper surjective}, this in turn implies that $G$ acts properly and cocompactly on $\theta(\partial X)$. Hence the lemma.
\end{proof}
We are now ready to prove the following.

\begin{proposition}\label{prop:uniform-convergence-action}
A hyperbolic TDLC group acts as a uniform convergence group on its Gromov boundary.
 \end{proposition}
 \begin{proof}
    Let $G$ be a hyperbolic TDLC group with Cayley-Abels graph $X$. By \cite[Lemma 3.3]{quasiactions}, the action of $G$ on $X$ is proper and cocompact. Hence, by Lemma \ref{uc action non-discrete}, $G$ acts on the Gromov boundary $\partial X$ as a uniform convergence group.
 \end{proof}
 
\begin{definition}
    [Conical limit point] \label{defn:conical-limit-point-1}Let $G$ be a TDLC group acting as a convergence group on a compact metrizable space $X$. A point $x\in X$ is said to be {\em conical} if there exists a sequence $\{g_n\}_n$ in $G$ and two points $y\neq z$ in $X$ such that $g_nx\to y$ and $g_nx'\to z$ for all $x'\neq x$.
\end{definition}
The following is a characterisation of conical limit points in terms of uniform convergence groups.
\begin{theorem}\cite[Theorem A(A$_2$)]{mathieu-dennis-locally-compact-convergence}\label{theorem:uniform-convergence-charaterisation}
    Let $G$ be a TDLC group acting on a compact metrizable space $X$ as a convergence group. Then, $G$ acts on $X$ as a uniform convergence group if and only if every point of $X$ is conical. 
\end{theorem}
%%%%%%%%%%%%%%%%%%%%%%%%%%%%%%%%%%%%%%%%%%%%%%%%%%%%%%%%%%%%%%%%%%%%%%%%%%%%%%%%%%%%%%%%%%%%%%%%%%%%%%%%%%%%%%%%%%%%%%%%%%%%%%%%%%%%%%%
\subsection{Cannon-Thurston maps}\label{subsection:CT maps}
Suppose $f:X\to Y$ is a qi embedding between Gromov hyperbolic spaces $X$ and $Y$. Then, by \cite[Theorem 3.9, III.H]{bridson-haefliger}, $f$ gives rise to a continuous map from the Gromov boundary of $X$ to the Gromov boundary of $Y$. It is natural to ask if non-qi embeddings could also induce continuous maps between Gromov boundaries of hyperbolic spaces. This leads to the notion of Cannon-Thurston (CT) maps. Let $X$ be a hyperbolic space and $Y$ be a subspace of $X$ which is hyperbolic with respect to the induced length metric from $X$. We say that the pair $(Y,X)$ admits a {\em Cannon-Thurston map} if the inclusion $Y\to X$ extends to a continuous map from the Gromov boundary of $Y$ to the Gromov boundary of $X$.

A sufficient condition for the existence of a Cannon-Thurston map is given in the form of the following lemma \cite[Lemma 2.1]{mitra-trees}, which is popularly known as Mitra's criterion for the existence of CT maps.
\begin{lemma}[Mitra's criterion]\label{lemma:mitra-criterion}
    A Cannon-Thurston map exists for the pair $(Y,X)$ if the following condition is satisfied: given $y_{0} \in Y$, there exists a non-negative function $M(N)$ such that $M(N) \to \infty$ as $N \to \infty$ and for all geodesic segments $\lambda$ lying outside an $N$-radius ball around $y_{0}$ in $Y$ any geodesic segment in $X$ joining the endpoints of $\lambda$ lies outside the $M(N)$-radius ball around $y_{0}$ in $X$.
\end{lemma}
Now, we define Cannon-Thurston maps in the context of hyperbolic TDLC groups. Let $G$ be a hyperbolic TDLC group and $H$ be an open hyperbolic TDLC subgroup of $G$. Let $\iota:\Gamma_H\to\Gamma_G$ be the embedding between Cayley-Abels graphs of $H$ and $G$.
\begin{definition}
    We say that the pair $(H,G)$ admits a {\em Cannon-Thurston map} if the map $\iota$ extends to a continuous map between the Gromov boundaries of $H$ and $G$.
\end{definition}
In the above definition, if $H$ is compact then its Gromov boundary is empty, and thus the CT map for the pair $(H,G)$ always exists. Also, if $1\to H\to G\to Q\to 1$ is a short exact sequence of compactly generated TDLC groups such that $H$ is compact and $Q$ is hyperbolic, then, by Remark \ref{remark:compact-extension}, $G$ is hyperbolic and the pair $(H,G)$ admits a CT map.
%%%%%%%%%%%%%%%%%%%%%%%%%%%%%%%%%%%%%%%%%%%%%%%%%%%%%%%%%%%%%%%%%%%%%%%%%%%%%%%%%%%%%%%%%%
\section{Quasiconvex subgroups of hyperbolic TDLC groups and their limit sets}\label{section:quasiconvexity}
In \cite{gromov-hypgps}, Gromov introduced the notion of quasiconvex subsets of a hyperbolic metric space. A subset $Y$ of a geodesic metric space $Z$ is said to be {\em quasiconvex} if there exists a constant $D\geq 0$ such that geodesics in $Z$ joining a pair of elements of $Y$ lie in $N_D(Y)$. Motivated by that, we define the notion of quasiconvexity for TDLC groups.
\begin{definition}[Quasiconvexity]\label{defn:quasiconvex-subgroup}
Let $H$ be a subgroup of a hyperbolic TDLC group $G$ with Cayley-Abels graph $X$. Then, $H$ is said to be {\em quasiconvex} if some $H$-orbit of a vertex of $X$ is quasiconvex.
\end{definition}
By Remark \ref{rem:canonical-cayley-abel-graph}, we can take $X$ to be equal to $\Gamma_G$ as constructed in Subsection \ref{subsection:2.2}. In general, like in the discrete setting, quasiconvexity of a subgroup depends on the chosen Cayley-Abels graph. However, in the presence of hyperbolicity, the following lemma shows that this choice is irrelevant.

\begin{lemma}\label{lemma:qc-independence}
    We have the following:
    \begin{enumerate}
    \item Definition \ref{defn:quasiconvex-subgroup} does not depend on the choice of an orbit of $H$.
        
        \item Quasiconvexity of H does not depend on the choice of a Cayley-Abels graph.
        
    \end{enumerate}
\end{lemma}

\begin{proof}
    (1) Suppose $Hv$ is $D$-quasiconvex in $X$ for some vertex $v\in V(X)$ and $D\geq 0$. Let $Hv'$ be another orbit of $H$ for some vertex $v\neq v'\in V(X)$. Let $\alpha$ be a geodesic in $X$ joining $h_{1}v'$ and $h_{2}v'$ for some $h_{1},h_{2} \in H$. Since $d_X(h_iv,h_iv')=d_X(v,v')=L$ (say), the Hausdorff distance of $Hv$ and $Hv'$ is at most $L$. If $\beta$ be a geodesic joining $h_{1}v$ and $h_{2}v$ in $X$, then, by Theorem \ref{stability of quasigeodesics},  the Hausdorff distance of $\alpha$ and $\beta$ is at most $L'$ for some $L'=L'(\delta,L)$, where $\delta$ is a hyperbolicity constant of $X$. Since $\beta$ lies in a $D$-neighborhood of $X$, $\alpha$ lies in a $(D+L+L')$-neighborhood of $Y$. Hence, $Hv'$ is quasiconvex in $X$.

    (2) Let $\Gamma_{G}= \Gamma_{G}(K,U,A)$ be a Cayley-Abels graph of $G$ and let $\{hU: h \in H\}$ be $D$-quasiconvex in $\Gamma_{G}$. Let $\Gamma'_{G}= \Gamma_{G}(K',U',A')$ be another Cayley-Abels graph of $G$. Let $\alpha$ be a geodesic in $\Gamma_{G}'$ joining $hU'$ and $h'U'$ for some $h,h' \in H$. By Theorem \ref{theorem:existenc-qi-invariance-cayley-abel-graph}(2), we know that $\Gamma_{G}$ and $\Gamma_{G}'$ are quasiisometric. Let $\psi$ be the quasiisometry from $\Gamma'_{G}$ to $\Gamma_{G}$ which is defined on the vertex sets of $\Gamma_{G}$ and $\Gamma_{G}'$ by $\psi(gU'):=gU$ (see \cite[Theorem 2.7$^{+}$]{kron-moller}). Therefore, $\psi(\alpha)$ is a quasigeodesic in $\Gamma_{G}$ joining $hU$ and $h'U$. Since $\{hU: h \in H\}$ is $D$-quasiconvex in $\Gamma_{G}$, $\psi(\alpha)$ lies in a $D$-neighborhood of $\{hU: h \in H\}$. Then $\alpha$ is in a $D'$-neighborhood of $\{hU': h \in H\}$, where the constant $D'$ depends on $D$ and the parameters of $\psi$. Hence $\{hU: h \in H\}$ is $D'$-quasiconvex in $\Gamma'_{G}$.
\end{proof}
Now, we record a few facts about quasiconvex subgroups of hyperbolic TDLC groups, which are analogous in the discrete setting.

\begin{lemma} \label{lemma:properties-of-qc} Suppose $G$ is a hyperbolic TDLC group and $H$ is an open subgroup of $G$. Then we have the following:
\begin{enumerate}
    \item If $H$ is quasiconvex in $G$, then $H$ is compactly generated.
    \item If $H$ is compactly generated and $\Gamma_{H}$ is quasiisometrically embedded in $\Gamma_{G}$, then $H$ is quasiconvex in $G$.
    \item If $H$ is quasiconvex, then there is a quasiisometric embedding from a Cayley-Abels graph of $H$ to a Cayley-Abels graph of $G$. In particular, $H$ is a hyperbolic TDLC group.
\end{enumerate}
\end{lemma}
\begin{proof}
(1) Let $\Gamma_G$ be the Cayley-Abels graph of $G$. Suppose the orbit of $U\in V(\Gamma_G)$ is $D$-quasiconvex in $\Gamma_G$ for some $D\geq 0.$ For $h \in H$, let $\gamma$ be a geodesic in $\Gamma_G$ joining $U$ and $hU$. Let $U,a_{1}U,\cdots,a_{1}\cdots a_{n}U=hU$ be the vertices of $\gamma$. Since $H$ is $D$-quasiconvex, for $i=1,\cdots,n$, there exists $h_{i}'\in H$ such that $d_{\Gamma_{G}}(a_{1} \cdots a_{i}U, h_{i}'U) \leq D$. Now, for $1\leq i\leq n$, we choose a word $b_{i}$ in $\langle A\rangle$ of length at most $D$ so that $h_{i}'U=a_{1}\cdots a_{i}b_{i}U$ (here $b_{n}=1$). This implies that $h_{1}'= a_{1}b_{1}u_{1}$ for some $u_{1} \in U$, and $(h'_{i-1})^{-1}h'_{i}=u_{i-1}^{-1}b_{i-1}^{-1}a_{i}b_{i}u_{i}$ for some $u_{i} \in U$ and $2\leq i\leq n$. Taking $h_{1}=h_{1}'$ and $h_{i}=(h'_{i-1})^{-1}h_{i}'$ for $2\leq i\leq n$,  we see that $hU=h_{1} \cdots h_{n}U$. Note that $h_{i}$ lies in $(2D+1)$-radius ball around $U$ in $\Gamma_{G}$. 
    %since $UAU=. 
    This implies that $h=h_{1}\cdots h_{n}u$ for some $u \in U$, and therefore $u\in U\cap H$. It follows that $H$ is generated by the union of $B_{2D+1}(U)\cap H$ and $U\cap H$, which is clearly compact. Hence, $H$ is compactly generated.
    
    (2) Suppose $G$ is $\delta$-hyperbolic and $\iota: \Gamma_{H} \to \Gamma_{G}$ is a $(\lambda,\epsilon)$-quasiisometric embedding for some $\delta\geq 0,\lambda\geq 1,\epsilon\geq 0$. Let $v \in V(\Gamma_{H})$, and $\alpha$ be a geodesic joining $h\iota(v)$ and $h'\iota(v)$ in $\Gamma_G$. Let $\beta$ be a geodesic joining $hv$ and $h'v$ in $\Gamma_{H}$. Since $\iota$ is a quasiisometric embedding, $\iota(\beta)$ is a quasigeodesic joining $h\iota(v),h'\iota(v) \in \Gamma_{G}$. By Theorem \ref{stability of quasigeodesics}, there exists $R=R(\delta,\lambda,\epsilon)$ such that the Hausdorff distance between $\alpha$ and $\iota(\beta)$ is at most $R$. This implies that $H$-orbit of $\iota(v)$ is $R$-quasiconvex in $\Gamma_G$. Hence, $H$ is quasiconvex in $G$.
    
    (3) Let $H$ be a quasiconvex subgroup of $G$. By (1), $H$ is compactly generated. Let $\Gamma_H$ and $\Gamma_G$ be the Cayley-Abels graphs of $H$ and $G$, respectively and let $\iota:\Gamma_H\to \Gamma_G$ be the embedding. Since any two Cayley-Abels graphs of a compactly generated TDLC group are quasiisometric (Theorem \ref{theorem:existenc-qi-invariance-cayley-abel-graph}), it is enough to show that $\iota$ is a qi embedding. 
   By Lemma \ref{lemma:proper-embedding}, we see that $\iota$ is a proper embedding. Since $\iota(\Gamma_H)$ is quasiconvex in $\Gamma_G$, by \cite[Lemma 1.93]{kapovich-sardar-book}, $\iota$ is a quasiisometric embedding. Finally, by \cite[Theorem 1.9, III.H]{bridson-haefliger}, it follows that $\Gamma_H$ is hyperbolic, and thus $H$ is hyperbolic. This completes the proof of the lemma.
\end{proof}
    
\begin{lemma}
    A finite index subgroup $H$ of a hyperbolic TDLC group $G$ is compactly generated and quasiconvex.
\end{lemma}
\begin{proof}
    Let $\Gamma_G$ be the Cayley-Abels graph of $G$. Then, $G$ acts on $\Gamma_{G}$ properly and cocompactly by  \cite[Lemma 3.3]{quasiactions}. We now show that $H$ also acts on $\Gamma_{G}$ properly and cocompactly. Clearly, $H$ acts properly on $\Gamma_G$. Since the action of $G$ on $\Gamma_{G}$ is cocompact, there is some compact set $K\subset \Gamma_{G}$ such that $GK= \Gamma_{G}$. Since $[G:H]<\infty$, for some $n\in\mathbb N$, we can write $G= \bigcup_{i=1}^{n} Hg_{i}$, where $g_{i} \in G.$ Therefore, $\Gamma_{G}= (\bigcup_{i=1}^{n} Hg_{i})K= H(\bigcup_{i=1}^{n} g_{i}K)$. Since each $g_{i}K$ is compact and a finite union of compact sets is compact, $\bigcup_{i=1}^{n} g_{i}K$ is compact. Therefore, the action of $H$ on $\Gamma_{G}$ is compact. Thus, by Theorem \ref{milnor-svarc-lemma}, $H$ is compactly generated and, for any fixed vertex $v$ of $\Gamma_{G}$, the orbit map $H \rightarrow \Gamma_{G}$, $h \mapsto hv$, is a quasiisometry (here $H$ has the word metric coming from the compact generating set). Since the orbit map is a quasiisometry, the orbit $\{hv:h\in H\}$ is coarsely dense in $\Gamma_G$, and hence it is quasiconvex in $\Gamma_G$. Thus, $H$ is quasiconvex in $G$.
\end{proof}
For discrete hyperbolic groups, the following is well-known, see for example \cite{short}. With appropriate modifications, we prove it for TDLC groups. 
\begin{lemma}\label{lemma:intersection-of-qc-is-qc}
    Let $G$ be a hyperbolic TDLC group. Then the intersection of finitely many open quasiconvex subgroups is open quasiconvex. 
\end{lemma}
\begin{proof}
    By induction, it is sufficient to prove the intersection of two quasiconvex subgroups of $G$ is quasiconvex. Let $H$ and $K$ be two open $C$-quasiconvex subgroups of $G$ for some $D\geq 0$. Let $\Gamma_G=\Gamma_G(K,U,A)$ be the Cayley-Abels graph of $G$. Suppose $H\cap K$ is not quasiconvex. Then, the orbit $\mathcal O:=\{gU:g\in H\cap K\}$ is not quasiconvex in $\Gamma_G$. This implies that there exists a sequence $\{g_iU\}\in\mathcal O$ such that some geodesic path $\gamma_i$ in $\Gamma_G$ joining $U$ and $g_iU$ does not lie in $r_i$-neighborhood of $\mathcal O$, where $r_i\to \infty.$ Thus, we have a vertex $v_i$ on $\gamma_i$ such that $B_{r_i}(v_i)$ does not intersect $\mathcal O$. Since $H$ and $K$ are quasiconvex, there exist $h_i\in H$ and $k_i\in K$ such that $h_i^{-1}k_iU$ belongs to a $2D$-radius ball around $U$. Since balls of finite radius in $\Gamma_G$ contain only finitely many vertices, $h_i^{-1}k_iU=gU$ for some $g\in G$ for infinitely many $i$. Thus, there exists $u_i\in U$ such that $h_i^{-1}k_i=gu_i$ for infinitely many $i$. Since $U\cap K$ is open and $U$ is compact, $[U:U\cap K]<\infty.$ Therefore, there exists $u_0\in U$ and $k_{0i}\in U\cap K$ such that $u_i=u_0k_{0i}$ for infinitely many $i$. By taking $k_i'=k_ik_{0i}^{-1}$, we see that $h_i^{-1}k_i'=gu_0$ for infinitely many $i$. Choose $r_i>D+|a_1|$, where $|a_1|$ is the distance from $U$ to $a_1U$ in $\Gamma_G$. Now, $h_1h_i^{-1}=k_1'k_i'^{-1}\in H\cap K$ and $v_i$ is $(D+|a_1|)$-close to $h_1h_i^{-1}$, a contradiction. 
\end{proof}

\begin{remark}\label{remark:normalizer-centralizer}
In contrast with the discrete setting, we note the following:
    \begin{enumerate}
        \item The normalizer of an open quasiconvex subgroup in a hyperbolic TDLC group need not be quasiconvex.
        \item The index of a quasiconvex subgroup of a hyperbolic TDLC group in its normalizer is not necessarily finite.
    \end{enumerate}
    \end{remark}
    
The above remark can be justified by the following example.
\begin{example}\label{example:normalizer-centralizer}
    Let $G=\mathbb{Z}\ltimes_p \mathbb{Q}_p$ denote the semidirect product of the discrete group $\mathbb{Z}$ with the additive group $\mathbb{Q}_p$ of $p$-adic numbers, where $\mathbb{Z}$ acts on $\mathbb{Q}_p$ by $n\cdot x=p^n x$. Then $G$ is a TDLC group with multiplication
    \[
        (n,x)\cdot (m,y)=(n+m,\;x+p^n y).
    \]
    By \cite[Example 3.7(2)]{combination-theorem}, $G$ is hyperbolic. Moreover, one checks that both the normalizer and centralizer of $\mathbb{Z}_p$ in $G$ are $\mathbb{Q}_p$. Since $\mathbb Q_p$ is not compactly generated, it is not quasiconvex by Lemma \ref{lemma:properties-of-qc}(1).
\end{example}
\subsection{Application to CT maps} 
A quasiconvex subgroup of a discrete hyperbolic group always admits a CT map (\cite[Theorem 3.9, III.H]{bridson-haefliger}). The following lemma is a non-discrete version of this fact.

\begin{lemma}\label{lemma:qc-implies-CT}
    Suppose $H$ is an open quasiconvex subgroup of a hyperbolic TDLC group $G$. Then, the pair $(H, G)$ admits a CT map.
\end{lemma}
\begin{proof}
      Let $\iota:\Gamma_H\to\Gamma_G$ be the embedding of Cayley-Abels graphs of $H$ and $G$, respectively. Since $H$ is quasiconvex in $G$, $\iota$ is a qi embedding by Lemma \ref{lemma:properties-of-qc}(3). Hence, by \cite[Theorem 3.9, III.H]{bridson-haefliger}, the CT map exists for the pair $(H, G)$.
\end{proof}
In the rest of this subsection, let $H$ be an open hyperbolic TDLC subgroup of a hyperbolic TDLC group $G$.  Let $\iota:\Gamma_H\to\Gamma_G$ be the embedding of Cayley-Abels graphs of $H$ and $G$, respectively. We show that the CT map for $\iota$ is injective if and only if $H$ is quasiconvex in $G$. For that, we need a little preparation from \cite{mahan-on-a-theorem-of-scott}. If $\lambda$ is a geodesic segment in $\Gamma_{H}$ then $\lambda^{r}$, a geodesic in $\Gamma_{G}$ joining the end-points of $\iota(\lambda)$, is called a geodesic realisation of $\lambda$. Consider sequences of geodesic segments $\lambda_{n} \subset \Gamma_{H}$ such that $U_{H} \in \lambda_{n}$ and $\lambda_{n}^{r} \cap B_U(n) = \phi$, where $B_U(n)$ is the ball of radius $n$ around $U \in V(\Gamma_{G})$ in $\Gamma_G$. Let $\mathcal{L}$ denote the set of all bi-infinite subsequential limits of all such sequences $\{\lambda_{n}\}$. Let $t_h$ denote the left translation by the element $h\in H$. Then, following \cite{mahan-on-a-theorem-of-scott}, we define the set of ending laminations for a hyperbolic subgroup $H$ of a hyperbolic TDLC group $G$ as follows:

\begin{definition}
    The set of {\em ending laminations} $\Lambda = \Lambda(H,G)$ is defined
by $\Lambda :=\{(p,q) \in \partial \Gamma_{H} \times \partial \Gamma_{H}: p\neq q \text{ and } p,q \text{ are the endpoints of } t_{h}(l) \text{ for some } l \in \mathcal L\}$.
\end{definition}

The idea of the proof of the following lemma is the same as the proof of \cite[Lemma 2.1]{mahan-on-a-theorem-of-scott}. For completeness, we include the proof.
\begin{lemma}
    $H$ is quasiconvex in $G$ if and only if $\Lambda= \phi$.
\end{lemma}
\begin{proof}
    Suppose $H$ is quasiconvex in $G$. Let $\lambda$ be a geodesic segment in $\Gamma_{H}$ and $\lambda^{r}$ be its geodesic realisation in $\Gamma_{G}$. Since $G$ is quasiconvex, $\iota:\Gamma_{H} \to \Gamma_{G}$ is a qi embedding by Lemma \ref{lemma:properties-of-qc}(3). Therefore $\iota(\lambda)$ is a quasigeodesic in $\Gamma_{G}$ joining the endpoints of $\lambda^{r}$. By Theorem \ref{stability of quasigeodesics}, they are Hausdorff close in $\Gamma_{G}$ since $G$ is hyperbolic. Therefore, there does not exist a sequence of geodesic segments $\{\lambda_{i}\} \subset \Gamma_{H}$ such that $U_{H} \in \lambda_{i}$ and $\lambda_{i}^{r} \cap B_U(i) = \phi$. Thus, $\Lambda= \phi$. 

    Conversely, suppose $H$ is not quasiconvex in $G$. Then there exists a sequence of geodesic segments $\{\lambda_{i}\} \subset \Gamma_{H}$ and $h_iU_H\in\lambda_i$ such that $\lambda_{i}^{r} \cap B_{h_iU}(i) = \phi$. Translating by $h_{i}^{-1}$ and taking subsequential limits, we have $\mathcal L \neq \phi$. Hence $\Lambda \neq \phi$.
\end{proof}

Suppose the pair $(H,G)$ admits a CT map $\partial \iota:\partial H\to\partial G.$ Define a set $\Lambda_{CT} := \{(p,q) \in \partial\Gamma_{H} \times \partial\Gamma_{H}: p\neq q \text{ and } \partial{\iota}(p) =\partial{\iota}(q)\}.$ The proof of the following lemma follows from the proof of \cite[Lemma 2.2]{mahan-on-a-theorem-of-scott}. Hence, we skip its proof.

\begin{lemma}
    $\Lambda= \Lambda_{CT}$.       \qed
\end{lemma}
By combining the previous two lemmas, we have the following.
\begin{proposition}\label{prop:CT-injective-iff-qc}
    The CT map $\partial \iota$ is injective if and only if $H$ is quasiconvex in $G$.  \qed
\end{proposition}

\subsection{Limit set of quasiconvex subgroups}
\begin{definition}[Limit set]\label{defn:limit-set}
    Let $G$ be a hyperbolic TDLC group with Cayley-Abels graph $X$. The {\em limit set} of a subgroup $H$ of $G$ is denoted by $\Lambda(H)$, and it is defined as the set of limit points of (any) some orbit of $H$ in $X$, i.e.
    $$\Lambda(H):=\{\xi\in \partial X: \text{there exists a sequence in the orbit of }H \text{ converging to } \xi\}.$$
\end{definition}
\begin{definition}\cite{swenson}
   Let $G$ be a hyperbolic TDLC group with Cayley-Abels graph $X$. A limit point $\xi$ of a subgroup of $H$ of $G$ is said to be {\em conical} if, for all rays $\alpha$ representing $\xi$, there exists a constant $r\geq 0$ such that $N_r(\alpha)\cap Hv\neq \phi,$ where $Hv$ denotes the $H$-orbit of $v$ in $X$.
\end{definition}
In \cite{swenson}, Swenson proved that a subgroup of a discrete hyperbolic group is quasiconvex if and only if every limit point is conical. Using a completely different technique, we generalise it in our setting. 
\begin{lemma}\label{lemma:char-qc-in-terms-conical}
    Let $G$ be a hyperbolic TDLC group and let $H$ be an open subgroup of $G$. Then, $H$ is quasiconvex in $G$ if and only if every limit point of $H$ is conical.
\end{lemma}
\begin{proof}
    It follows from the definition of conical limit points that if $H$ is quasiconvex in $G$, then every limit point of $H$ is conical. Conversely, suppose $\xi$ is a conical limit point. Then, $\xi$ is a conical limit point as in Definition \ref{defn:conical-limit-point-1} for $H\curvearrowright\partial G$. Since $\Lambda(H)$ is a closed subset of $\partial G$, $\xi$ is conical for $H\curvearrowright \Lambda(H)$. Thus, by Theorem \ref{theorem:uniform-convergence-charaterisation}, $H$ acts as a uniform convergence group on $\Lambda(H).$ This implies that $H$ is a hyperbolic group and $\Lambda(H)$ is $H$-equivariantly homeomorphic to the Gromov boundary of $H$. Therefore, an injective CT map exists for $(H,G)$. Hence, by Proposition \ref{prop:CT-injective-iff-qc}, $H$ is quasiconvex in $G$.
\end{proof}
Two quasiconvex subgroups $H$ and $K$ of a hyperbolic TDLC group $G$ are said to satisfy {\em limit set intersection property} if $\Lambda(H\cap K)=\Lambda(H)\cap\Lambda(K).$ Now, we prove the following, generalising \cite[Theorem 13]{swenson}. 
\begin{proposition}\label{prop:limit-set-intersection-property}
    Open quasiconvex subgroups of hyperbolic TDLC groups satisfy the limit set intersection property.
\end{proposition}
\begin{proof}
    Let $G$ be a hyperbolic TDLC group with open quasiconvex subgroups $H$ and $K$. By Lemma \ref{lemma:intersection-of-qc-is-qc}, $H\cap K$ is quasiconvex in $G$. Clearly, $\Lambda(H\cap K)\subseteq \Lambda(H)\cap\Lambda(K).$ 
    
    For the reverse inclusion, let $\xi\in\Lambda(H)\cap\Lambda(K).$ By Lemma \ref{lemma:char-qc-in-terms-conical}, it is enough to show that $\xi$ is a conical limit point of $H\cap K$. Fix a base point $aU\in \Gamma_G$ and let $\alpha$ be a geodesic ray based at $aU$ representing $\xi$. Since $K$ is quasiconvex, there is $D\geq0$ such that $\alpha$ lies in the $D$-neighbourhood of $KaU$.
 Since $\xi$ is a conical limit point of $H$, there exists $N'>0$ and a sequence $h_i\in H$ such that $d(h_i(aU),\alpha)\le N'$ and $d(h_i(aU),h_j(aU))>4N'$ for $i\ne j$. For each $i$ choose $k_i\in K$ with $d(k_i(aU),h_i(aU))\le D+N'$. Let $N:=D+N'$ and let 
$C:=\{g\in G: gB_N(aU)\cap B_N(aU)\neq\phi\}.$
Since the action of $G$ on $X$ is proper, the closure $\overline{C}$ of $C$ is compact, and thus we have that $U$ has finitely many cosets in $\overline{C}$. For all $i$, 
$h_i^{-1}k_iB_N(aU)\cap B_N(aU)\neq\phi$, and therefore $h_i^{-1}k_i\in C$. By passing to a subsequence, we may assume there is a fixed $g\in \overline{C}$ such that $h_i^{-1}k_i U=gU$ for all $i$. For each such $i$, there are $u_i\in U$ with $h_i^{-1}k_i=gu_i$. Since $U$ is compact and $U\cap K$ is open, $[U:U\cap K]<\infty$, and so there exists a fixed $u_0\in U$ and $k_{0i}\in K$ with $u_i=u_0k_{0i}$ for infinitely many $i$, giving $h_i^{-1}k_i=gu_0k_{0i}$. Writing $k_i':=k_ik_{0i}^{-1}$, we have $h_i^{-1}k_i'=gu_0$ for infinitely many $i$, and hence $h_i h_1^{-1}=k_i'k_1'^{-1}\in H\cap K$ for all $i$. Finally, since $h_ih_1^{-1}aU$ lies in $(N+|h_1^{-1}a|)$-neighborhood of $\alpha$ where $|h_1^{-1}a|$ denotes the distance from $aU$ to $h_1^{-1}aU$ in $\Gamma_G$, it follows that $\xi$ is a conical limit point of $H\cap K$. This completes the proof of the proposition.
\end{proof}

\begin{proposition}\label{prop:finite-index}
Let $G$ be a TDLC group acting properly and
cocompactly on a locally compact Hausdorff space $Y$. Let $H < G$ be an open subgroup that also acts cocompactly on $Y$.
Then $H$ has finite index in $G$.
\end{proposition}

\begin{proof}
Since the action of $H$ on $Y$ is cocompact, there exists a compact set
$K \subset Y$ such that
$Y = \bigcup_{h \in H} hK.$
Let $g \in G$ be arbitrary. Then $gK$ is a compact subset of $Y$, and hence
there exists $h \in H$ such that
$gK \cap hK \neq \phi.$
Equivalently,
$h^{-1}gK \cap K \neq \varnothing.$
Define
$C := \{\gamma \in G \mid \gamma K \cap K \neq \phi\}.$
The above argument shows that every $g \in G$ can be written as $g = h \gamma \text{ with }
h \in H,\ \gamma \in C,$ and hence $G = H C.$ Since the action of $G$ on $Y$ is proper, the set $\overline{C}$ is compact in $G$.
Since $H$ is open in $G$, the quotient $G/H$ is discrete.
The image of $\overline{C}$ in $G/H$ is compact and discrete, and therefore finite. It follows that $G/H$ is finite, and thus $[G:H] < \infty$.
\end{proof}
\begin{corollary}
    Let $G$ be a hyperbolic TDLC group and let $H$ be an open quasiconvex subgroup of $G$ such that $\Lambda(H)=\partial G$. Then, $H$ has finite index in $G$.
\end{corollary}
\begin{proof}
    Let $X$ be a Cayley-Abels graph of $G$. Then, $G$ acts properly and cocompactly on $X$. Since $\Lambda(H)=\partial G$, weak quasiconvex hull of $\Lambda(H)$ (the union of all geodesic lines in $X$ joining a pair of limit points of $H$) is equal to $X$. Since $H$ is quasiconvex, $H$ acts cocompactly on the weak quasiconvex hull of $\Lambda(H)$ and thus it acts cocompactly on $X$. Now, the corollary follows from Proposition \ref{prop:finite-index}.
\end{proof}

%%%%%%%%%%%%%%%%%%%%%%%%%%%%%%%%%%%%%%%%%%%%%%%%%%%%%%%%%%%%%%%%%%%%%%%%%%%%%%%%%%%%
\subsection{Height of a subgroup of a TDLC group}
The notion of the height of a subgroup of a discrete group was introduced by Gitik-Mitra-Rips-Sageev \cite{GMRS}. In the same paper, they proved that quasiconvex subgroups of a discrete hyperbolic group have finite height. We adapt this notion in our setting and prove the same result for hyperbolic TDLC groups.

\begin{definition}\label{defn:height}
Let $G$ be a TDLC group and $H$ be a subgroup of $G$. The {\em height} of $H$ in $G$, denoted by ${\rm ht}_G(H)$, is the maximum number $n$ for which there exist $n$ distinct cosets $g_1H, g_2H,\cdots,g_nH$ such that the intersection $\bigcap_{i=1}^ng_i^{-1}Hg_i$ is non-compact. If $H$ is compact, then we define ${\rm ht}_G(H)$ to be $0$.
%If no such number exists, the height is infinite.
\end{definition}
Note that $ht_G(H)=1$ if and only if $H$ is weakly malnormal.

\begin{proposition}\label{prop:qc-implies-finite-height}
    If $H$ is an open quasiconvex subgroup of a TDLC group $G$. Then ${\rm ht}_G(H)$ is finite.
\end{proposition}

\begin{proof}
 Suppose $H$ is $D$-quasiconvex for some $D\geq 0$ and $X$ is a Cayley-Abels graph of $G$. Let $g_1H, g_2H,\cdots, g_mH$ be $m$ distinct cosets such that $\bigcap_{i=1}^mg_i^{-1}Hg_i$ is non-compact. We will show that there exists an $n_0\in\mathbb N$ (independent of $m$) such that $m\leq n_0$. By Lemma \ref{lemma:intersection-of-qc-is-qc}, $\bigcap_{i=1}^mg_i^{-1}Hg_i$ is quasiconvex, and thus it is a hyperbolic TDLC group by Lemma \ref{lemma:properties-of-qc}. Since $\bigcap_{i=1}^mg_i^{-1}Hg_i$ is non-compact, the Gromov boundary of $\bigcap_{i=1}^mg_i^{-1}Hg_i$ contains at least two points \cite[\S 3.1,\S 8.2]{gromov-hypgps} (see also \cite[p.2904]{caprace-mood-amenable-hyp}. Let $l$ be a geodesic line in $X$ joining two points in $\Lambda(\bigcap_{i=1}^mg_i^{-1}Hg_i)$. By Proposition \ref{prop:limit-set-intersection-property}, we have $\Lambda(\bigcap_{i=1}^mg_i^{-1}Hg_i)=\cap_{i=1}^m\Lambda(g_i^{-1}Hg_i)=\cap_{i=1}^m\Lambda(g_i^{-1}H).$ Let $v\in l$ be a vertex. Since $g_i^{-1}H$ is also $D$-quasiconvex, there exists $D'=D'(D,\delta)$ such that $v$ belongs to the $D'$-neighborhood of $g_i^{-1}H$ for $1\leq i\leq m$.
Since the balls of finite radius in $X$ have finitely many vertices, there is a finite number $n_0$ of left cosets of $H$ which can intersect this ball; and hence the ${\rm ht}_G(H)$ is at most $n_o$, as required.
\end{proof}
%%%%%%%%%%%%%%%%%%%%%%%%%%%%%%%%%%%%%%%%%%%%%%%%%%%%%%%%%%%%%%%%%%%%%%%%%%%%%%%%%%%%%%%%%%%%%%%%%%%%%%%
\section{Combination of convergence TDLC groups}\label{section:construction-of-Gromov-boundary}
This section aims to give a proof of Theorem \ref{theorem:acyl-combination-theorem}. We follow the technique developed by Dahmani in \cite{dahmani-comb}. As we have mentioned in Remark \ref{remark:sufficient-to-prove-amalgam-hnn}, it is sufficient to consider the case of an amalgamated free product and an HNN extension. We start this section with the following fact, which we will use later in the section.

\begin{lemma}\label{prop:dahmnai-1.8-prop} Let $G$ be a hyperbolic TDLC group and let $H$ be an open quasiconvex subgroup of $G$. If $\{g_n\}$ is a sequence of elements of $G$ all in distinct cosets of $H$, then there is a subsequence $\{g_{n_k}\}$ of $\{g_n\}$ such that $g_{n_k}\Lambda(H)$ converges uniformly to a point.
\end{lemma}
\begin{proof}
    Let $\Gamma_G=\Gamma_G(K,U,A)$ be the Cayley-Abels graph of $G$. Since $H$ is quasiconvex in $G$, the $H$-orbit of $U$ is quasiconvex in $\Gamma_G$. This implies that, for each $n$, the $g_nH$-orbit of $U$ is quasiconvex in $\Gamma_G$. Note that the distance from $U$ to the $g_nH$-orbit of $U$ tends to infinity as $n\to\infty.$ Now, the lemma follows from a fact on the visual diameter of quasiconvex subsets of a hyperbolic space, see \cite[Lemma 2.4]{GMRS}.
\end{proof}

Let $G=A\ast_C B$ or $A\ast_C$, where $A, B$ are hyperbolic TDLC groups and $C$ is an open quasiconvex subgroup of $A$ and $B$. By Lemma \ref{lemma:properties-of-qc}, $C$ is a hyperbolic group, and, by Lemma \ref{lemma:qc-implies-CT}, there is a topological embedding from the Gromov boundary of $C$ to the Gromov boundaries of $A$ and $B$. Let $T$ be the Bass-Serre tree of $G$. For each $v\in V(T)$, we denote by $X_v$ the Gromov boundary of the $G$-stabilizer $G_v$ of $v$. Similarly, for each edge $e\in E(T)$, we denote by $X_e$ the Gromov boundary of the $G$-stabilizer $G_e$ of $e$. 
\subsection{Construction of the candidate space}
Let $\tau$ be a subtree of $T$ such that $\tau$ is a fundamental domain of $T$. Let $V(\tau)$ denote the set of vertices of $\tau$. Note that, for each vertex $v\in V(\tau)$, the group $G_v$ is a hyperbolic TDLC group. Therefore, $G_v$ acts as a uniform convergence group on $X_v$. We set $\Omega$ to be $G \times(\bigsqcup_{v\in V(\tau)}X_v)$ quotiented by the natural relation
$$(g_1,x_1) = (g_2,x_2) \text{ if }  \exists v\in V(\tau), x_i\in X_v,g_2^{-1}g_1 \in G_v, (g_2^{-1}g_1)x_1 =x_2 \text{ for } i=1,2.$$

Thus, $\Omega$ is the disjoint union of the Gromov boundaries of the stabilizers of the vertices of $T$. Also, for each $v\in V(\tau)$, the space $X_v$ naturally embeds in $\Omega$ as the image of $\{1\}\times X_v$. Hence, we identify it with its image. The group $G$ naturally acts on the left on $\Omega$. For $g\in G$, $gX_v$ is the compact metrizable space on which the vertex stabilizer $G_{gv}$ acts as a uniform convergence group. Each edge allows us to glue together the Gromov boundaries of the adjacent vertex stabilizers along the limit set of the stabilizer of the edge. By our assumption, each edge group is embedded as a quasiconvex subgroup in the adjacent vertex groups. Let $e$ = $(v_1,v_2)$ be the edge in $\tau$. Then, by our assumption, there exist $G_e$-equivariant embeddings $f_{e,v_i}:X_e \to X_{v_i}$ for $i=1,2$. Similar maps are defined by translating the edges in $T\setminus \tau$.

Now, we define an equivalence relation $\sim$ on $\Omega$ that is the transitive closure of the following: 

Let $v$ and $v'$ be vertices of $T$. The points $x \in X_v$ and $x' \in X_{v'}$ are equivalent in $\Omega$ if there is an edge $e$ between $v$ and $v'$ and a point $x_e \in X_e$ satisfying $x = f_{e,v}(x_e)$ and $x' = f_{e,v'}(x_e)$ simultaneously. Let $\Omega/_{\sim}$ be the quotient under this relation and let $\pi:\Omega \to \Omega/_{\sim}$ be the quotient map. We denote an equivalence class $[x]$ of an element $x\in \Omega$ by $x$ itself.

Let $\partial T$ be the (visual) boundary of the tree. We define $X$ as a set: 
$$X = \partial T \sqcup (\Omega/_{\sim})$$
 
\begin{definition}[Domain]\label{domain}
	For each $x \in \Omega/_{\sim}$, we denote the {\em domain} of $x$ by $D(x)$ and define $D(x):= \{v\in V(T) : x\in \pi(X_v)\}$. For $\xi\in\partial T$, the domain of $\xi$ is defined to be $\{\xi\}$ itself.
\end{definition}

\begin{proposition}\label{prop:domians-uniformly-bounded}
    For each $x\in \Omega/_\sim$, $D(x)$ is a locally finite uniformly bounded subtree of $T$.
\end{proposition}
\begin{proof}
    First of all, we show that the pointwise stabilizer of a finite subtree of $D(\xi)$ is non-compact. In fact, we show that if $\{v_1,\cdots,v_n\}$ is a set of vertices of a finite subtree of $D(\xi)$, then there exists a subgroup $H$ of $G$ which is embedded in $G_{v_i}$ as quasiconvex subgroup and $\xi\in \Lambda(H).$ We prove it by induction on $n$. If $n=1$, then we can take $H$ to be equal to the vertex stabilizer. Suppose the result is true for any finite subtree of $D(\xi)$ whose number of vertices is strictly less than $n$. Up to reindexing, let $v_n$ be the vertex of the subtree of valence $1$ and $v_{n-1}$ be the adjacent vertex joined by the edge $e$. Then, by inductive hypothesis, for $1\leq i\leq n-1$, let $H_{n-1}$ be the subgroup embedded in $G_{v_i}$ as a quasiconvex subgroup and $\xi\in \Lambda(H_{n-1})$. Since $v_n\in D(\xi)$, $\xi\in \Lambda(G_e).$ This implies that $\xi\in \Lambda(H_{n-1})\cap \Lambda(G_e)$ and thus, by Proposition \ref{prop:limit-set-intersection-property}, $\xi\in \Lambda(H_{n-1}\cap G_e)$. Taking $H$ gives us the desired result. Now, since the pointwise stabilizer of every finite subtree of $D(\xi)$ is non-compact, the diameter of $D(\xi)$ is bounded above by the acylindricity constant. Let $v$ be a vertex in $D(\xi)$. Suppose $\{e_n\}_n$ is an infinite sequence of edges in $D(\xi)$ adjacent to $v$. Then, up to passing to a subsequence, there exists an edge $e\in D(\xi)$ and a sequence $\{g_n\}_n$ in $G_v$ such that $\xi\in g_n\Lambda(G_e)$. This implies that $\bigcap g_n^{-1}G_eg_n$ is non-compact and we get to a contradiction by Proposition \ref{prop:qc-implies-finite-height}. 
\end{proof}
\subsection{Topology on $X$} We define a family $\{W_n(x)\}_{n\in \mathbb{N},x\in X}$ of subsets of $X$, which generates a topology on $X$. For a vertex $v\in T$ and an open subset $U$ of $X_v$, we define the subtree $T_{v,U}$ of $T$ as $$T_{v,U}:=\{w\in V(T) : X_e\cap U\neq \emptyset\}$$ 
where $e$ is the first edge of the geodesic $[v,w]\subset T$. For each vertex $v$ in $T$, let us choose $\mathcal B_v$, a countable basis of open neighborhoods of $X_v$. Without loss of generality, we can assume that, for all $v$, the collection of open subsets $\mathcal B_v$ contains $X_v$. Let $x\in\Omega/_\sim$ and let $D(x) = \{v_1,v_2,\dots\} = \{v_i\}_{i\in I\subset \mathbb{N}}$. For each $i\in I$, let $U_i$ be an element of $\mathcal B_{v_i}$ containing $x$ such that, for all but finitely many indices $i\in I, U_i=X_{v_i}$. We define the set $W_{(U_i)_{i\in I}}(x)$ as 
$$W_{(U_i)_{i\in I}}(x):=A \cup B\cup C$$
where $A,B,C$ are defined as follows.

{\em Description of $A,B, C$ in words:} The set $A$ is the collection of all points in the intersection of boundaries of the subtrees $T_{v_i, U_i}$. Note that $\partial (\bigcap_{i\in I}T_{v_i,U})=\bigcap_{i\in I}\partial T_{v_i,U}$. The set $B$ is the collection of all points that lie outside $\bigcup_{i\in I}X_{v_i}$ in $\Omega/_\sim$ and the domains of such points lie inside $\bigcap_{i\in I}T_{v_i,U_i}$. The set $C$ contains all elements in $\bigcup_{i\in I}X_{v_i}$ that belong to $\bigcap_{j\in J\subseteq I}U_j$ for some $J\subseteq I$.

In notations $A,B,C$ are defined as follows:
\begin{align*}
	A &= \bigcap_{i\in I} \partial T_{v_i,U_i} \\
	B &= \{y \in (\Omega/_\sim)\setminus(\bigcup_{i\in I}X_{v_i}):D(y) \subset \bigcap_{i\in I}T_{v_i,U_i}\}\\
	C &=\{y \in \bigcup_{i\in I}X_{v_i}: y\in\bigcap_{j\in J}U_j \text{ for some }J\subseteq I\}
\end{align*}

Observe that $A$, $B$, and $C$ are disjoint subsets of $X$. Next, we define neighborhoods around the points of $\partial T$. Let $\eta\in \partial T$ and fix a base point $v_0$ in $T$. Firstly, we define the subtree $T_m(\eta)$ as: $$T_m(\eta):=\{v\in V(T): l([v_0,w]\cap [v_0,\eta))>m\}$$ We set $$W_m(\eta):= \{\zeta \in X | D(\zeta) \subset\overline{T_m(\eta)}\}$$ where $\overline{T_m(\eta)}$ denotes the closure of $T_m(\eta)$ in $T\sqcup\partial T$. This definition does not depend on the choice of base point $v_0$, up to shifting the indices.

\begin{remark}
	The set $W_{(U_i)_{i\in I}}(x)$ is completely defined by the data of the domain of $x$, the data of a finite subset $L$ of $I$, and the data of an element of $\mathcal B_{v_l}$ for each index $l\in L$. Therefore, there are only countably many different sets  $W_{(U_i)_{i\in I}}(x)$, for $x\in \Omega'$, and $U_i\in \mathcal B_{v_i}, v_i\in D(x)$. For each $x$, we choose an arbitrary order and denote it by $W_m(x)$.
\end{remark}

Consider the smallest topology $\mathcal T$ on $X$ such that the family of sets $\{W_n(x):x\in X, n\in\mathbb{N}\}$ are open subsets of $X$. The proof of the following lemma follows from the proof of the lemmas in \cite[Subsection 2.4]{dahmani-comb}. Thus, we skip its proof.
\begin{lemma}\label{lemma:avoiding an edge}
    We have the following.
    \begin{enumerate}
    \item Let $x\in X$ and let $e$ be an edge of $T$ such that at least one vertex of $e$ is not in $D(x)$. Then, there exists $W_n(x)$ such that $X_e\cap W_n(x)=\emptyset$.
        \item The family of sets $\{W_n(x): n\in \mathbb{N} \text{ and }x\in X\}$ forms a basis for the topology $\mathcal T$ on $X$.
        \item $(X,\mathcal T)$ is a second countable Hausdorff regular space.
        \item $(X,\mathcal T)$ is compact.
    \end{enumerate}
\end{lemma}
The proof of the fact that $(X,\mathcal T)$ is compact is similar to the proof of \cite[Theorem 2.10]{dahmani-comb}. The only difference is that, in his proof, Dahmani used \cite[Proposition 1.8]{dahmani-comb}, and we can use Lemma \ref{prop:dahmnai-1.8-prop} in place of that proposition. It follows from the previous lemma that the topology $\mathcal T$ on $X$ is metrizable. Also, from the definition of neighborhood on $X$, we see that every point of $X$ is a limit point, and hence it is a perfect space. 

For a compact metric space $Z$, we denote by $\dim(Z)$ the topological dimension of $Z$. We end this subsection by noting the following, which is an upper bound on the topological dimension of $X.$ The proof of this fact is similar to \cite[Theorem 2.11]{dahmani-comb}. We leave it for the reader.

\begin{theorem}
    $\dim(X)\leq \max_{v,e}\{\dim(X_v),\dim(X_e)+1\}.$
\end{theorem}

\vspace{.2cm}

{\bf Proof of Theorem \ref{theorem:acyl-combination-theorem}:} From the proof of \cite[Lemma 3.1 and Lemma 3.2]{dahmani-comb}, we see that $G$ acts as a convergence group on $X$. Now, using Theorem \ref{theorem:uniform-convergence-charaterisation}, it is sufficient to show that every point of $X$ is conical. This follows from the proof of \cite[Lemma 3.4 and Lemma 3.5]{dahmani-comb}. By \cite[Lemma 2.1 and Corollary 2.8]{dahmani-comb}, the natural map from $X_v$ to $X$ is injective and continuous. Therefore, we have an injective CT map for the pair $(G_v,G).$ Hence, by Proposition \ref{prop:CT-injective-iff-qc}, $G_v$ is quasiconvex in $G$.    \qed

\vspace{.2cm}

We end this subsection with the following example that gives an example of a non-compact weakly malnormal subgroup of a hyperbolic TDLC group.

\begin{example}\label{example:malnormal-example}
    Let $D$ be an infinite discrete hyperbolic group and let $C$ be a totally disconnected compact non-discrete group. Then, with the product topology, $G:=D\times C$ is a non-discrete TDLC group. By Remark \ref{remark:compact-extension}, $G$ is hyperbolic. Let $M_D$ be an infinite weakly malnormal quasiconvex subgroup of $D$ (such a subgroup always exists when $D$ is torsion-free, see for example \cite{I-kapovich-malnormal-qc}) and let $M_C$ be a compact open subgroup of $C$. It is easy to check that $M=M_D\times M_C$ is a weakly malnormal open quasiconvex subgroup of $G$. Consider the double $G\ast_M \overline{G}$, where $\overline{G}$ is the isomorphic copy of $G$. Now, by Theorem \ref{theorem:acyl-combination-theorem}, $G\ast_M\overline{G}$ is a hyperbolic TDLC group, and $G,\overline{G}$ are quasiconvex subgroups of $G\ast_M \overline{G}$.
\end{example}
\subsection{Application to height in splitting of hyperbolic TDLC groups} The main goal of this subsection is to give a proof of Theorem \ref{theorem:height-splitting}. First of all, we prove the following general lemma, which we will use in the proof of the theorem.
\begin{lemma}\label{finite height implies acyl}
	 Let $(\mathcal G,\mathcal Z)$ be a finite graph of topological groups with fundamental group $G$. If all the edge groups of $(\mathcal G,\mathcal Z)$ have finite height in $G$, then the action of $G$ on the Bass-Serre tree of $(\mathcal G,\mathcal Z)$ is acylindrical.
\end{lemma}
\begin{proof}
	Since $(\mathcal G,\mathcal Z)$ is a finite graph of groups, there exists $n_0\in\mathbb{N}$ such that the height of each edge group of $(\mathcal G,\mathcal Z)$ is at most $n_0$. Let $T$ denote the Bass-Serre tree of $(\mathcal G,\mathcal Z)$. The $G$-stabilizer of a geodesic $\alpha$ in $T$ is the intersection of the $G$-stabilizers of the edges on $\alpha$. Suppose, if possible, the action of $G$ on $T$ is not acylindrical. Then, given $k\in\mathbb{N}$, there is a geodesic in $T$ with length bigger than $k$ and the $G$-stabilizer of the geodesic is non-compact. Choose $k$ sufficiently larger than $n_0$, and let $\beta$ be a geodesic segment of length bigger than $k$ with $G$-stabilizer non-compact. Since there are finitely many edge groups, the $G$-stabilizer of $\beta$ is contained in the intersection of more than $n_0$-conjugates of an edge group in $G$. This gives a contradiction as the height of each edge group is at most $n_0$.
\end{proof}
Proof of Theorem \ref{theorem:height-splitting}: If $C$ is quasiconvex in $G$, then $C$ has finite height by Proposition \ref{prop:qc-implies-finite-height}. Conversely, suppose $C$ has finite height in $G$. Then, by Lemma \ref{finite height implies acyl}, the action of $G$ on its Bass-Serre tree is acylindrical. Now, by Theorem \ref{theorem:acyl-combination-theorem}, $A$ and $B$ are quasiconvex in $G$. Hence, $C$ is quasiconvex in $G$. This completes the proof of the theorem.    \qed
%%%%%%%%%%%%%%%%%%%%%%%%%%%%%%%%%%%%%%%%%%%%%%%%%%%%%%%%%%%%%%%%%%%%%%%%%%%%%%%%%%%%%%%%%%%%
\section{Combination of locally quasiconvex hyperbolic TDLC groups}\label{section:combination-of-hyperbolic-TDLC-groups}
This section is devoted to giving a proof of Theorem \ref{theorem:combination-of-locally-qc}, its applications, and examples.

{Proof of Theorem \ref{theorem:combination-of-locally-qc}:} To prove the theorem, by induction, it is sufficient to consider the amalgamated free product and HNN extension case. In the rest of the proof, let $G$ denote either an amalgamated free product $A\ast_C B$  or an HNN extension $A\ast_C$, and $(\mathcal G,\mathcal Z)$ denote the corresponding graph of groups with fundamental group $G$. By Theorem \ref{theorem:acyl-combination-theorem}, $G$ is a hyperbolic TDLC group. Let $T$ be the Bass-Serre tree of $G$. Let $H$ be an open compactly generated subgroup of $G$. First, suppose $H$ fixes a vertex of $T$. Then, $H$ is conjugate to a subgroup of $A$ or $B$. Since $A$ and $B$ are locally quasiconvex and, by Theorem \ref{theorem:acyl-combination-theorem}, $A$ and $B$ are quasiconvex in $G$, $H$ is quasiconvex in $G$.  

Suppose now that $H$ does not fix a vertex of $T$. Then, $H\cap C$ is an open TDLC subgroup of $H$. Let $U$ be a compact open subgroup of $H\cap C$ and let $e$ be an edge of $T$ stabilized by $U$. Since $H$ is compactly generated, there exists a finite set $S$ such that $S\cup U$ generates $H$. Let $T'$ be the minimal connected subtree of $T$ containing the set of edges $\{se:s\in S\}$. Since $S$ is finite, $T'$ is a finite subtree of $T$. Since $S\cup U$ generates $H$, it follows that $T_H:=\bigcup_{h\in H}hT'$ is an $H$-invariant subtree of $T$. From the construction of $T_H$, it is immediate that $H$ acts on $T_H$ discretely (pointwise stabilizers of edges are open) and cocompactly on $T_H$. Thus, $H$ admits a finite graphs of groups decomposition $(\mathcal G,\mathcal Y)$, where $\mathcal Y$ is the finite quotient subgraph of $T_H$. Note that the vertex groups of $(\mathcal G,\mathcal Y)$ are the intersection of $H$ with the conjugates of $A$ and $B$ in $G$. Similarly, the edge groups $(\mathcal G,\mathcal Y)$ are the intersection of $H$ with the conjugates of $C$ in $G$. Now, by our assumption, the edge groups of $(\mathcal G,\mathcal Y)$ are compactly generated and $H$ is also compactly generated, the vertex groups of $(\mathcal G,\mathcal Y)$ are compactly generated \cite[Proposition 8.5]{arora-pedroja}. Since $A$ and $B$ are locally quasiconvex, the vertex and edge groups of $(\mathcal G,\mathcal Y)$ are quasiconvex subgroups of the corresponding vertex and edge groups of $(\mathcal G,\mathcal Z).$ Since the action of $G$ on $T$ is acylindrical, the action of $H$ on $T_H$ is also acylindrical. Thus, by Theroem \ref{theorem:acyl-combination-theorem}, $H$ is a hyperbolic group and the vertex groups of $(\mathcal G,\mathcal Y)$ are quasiconvex in $H$. Let $X_1$ be the Gromov boundary of $H$ as constructed in Section \ref{section:construction-of-Gromov-boundary}.

{\em Claim:} $X_1$ is the limit set of $H$, i.e. $X_1$ is the minimal closed $H$-invariant subset of $X$.

Once we have a proof of the claim, then we have an injective CT map for the pair $(H,G)$. Hence, by Proposition \ref{prop:CT-injective-iff-qc}, $H$ is quasicovex in $G$. This completes the proof of the theorem. 

{\em Proof of the claim:} First, note that $X_1$ is an $H$-invariant subset of $X$. We now show that $X_1$ is closed in $X$. To show that, we prove that the complement $X_1^c$ of $X_1$ in $X$ is open. Let $x\in X_1^c$. There are two possibilities. Suppose $x\in X_v$ for some vertex $v\in T_H^c$, where $T_H^c$ denotes the complement of $T_H$ in $T$. Let $\alpha$ be the geodesic connecting $T_H$ and the domain $D(x)$ of $x$. Then, there exists an edge $e$ on $\alpha$ such that both the vertices of $e$ are not in $D(x)$. Thus, by Lemma \ref{lemma:avoiding an edge}(1), there exists a neighborhood $W_n(x)$ such that $X_e\cap W_n(x)=\phi$. This implies that $W_n(x)\subset X_1^c$. Similarly, if $\eta\in \partial T$ and $\eta$ is not a boundary point of $T_H$ then, using Lemma \ref{lemma:avoiding an edge}(1), there exists a neighborhood $W_n(\eta)$ such that $W_n(\eta)\subset X_1^c$. Suppose $x\in X_u$ for some $u\in T_1$. For a vertex $u\in T_H$, let $X_1^u$ denote the Gromov boundary of the vertex group of $(\mathcal G,\mathcal Y)$ at the vertex $u$. Then, $x\notin X_1^u$ for all $u\in T_H$. Since $x\in X_1^c$, the intersection of $D(x)$ and the set of vertices of $T_H$ is finite. Let $D(x)=\{v_i\}_{i\in I}$ and $D(x)\cap V(T_H)=\{u_1,\dots,u_l\}$ for some $l\in\mathbb{N}$. Now, in each $X_{u_i}$, choose a neighborhood $U_i$ around $x$ which is disjoint from $X_1^{u_i}$ and such that, for any edge $e\in T_H$ adjacent to $u_i$, $X_1^{u_i}$ and $U_i$ does not intersect $X_e$ simultaneously (such a neighborhood is possible to choose using Lemma \ref{prop:dahmnai-1.8-prop}). Now, using these $U_i$ and the definition of neighborhoods, it follows that there exists a neighborhood $W_n(x)$ such that $W_n(x)\subset X_1^c$. Thus, $X_1$ is a closed subset of $X$. This implies that the limit set of $H$ is a subset of $X_1$. Since every point of $X_1$ is a limit point of $H$, the limit set of $H$ is equal to $X_1$. Hence the claim.    

For the converse, suppose $(\mathcal G,\mathcal Z)$ satisfies conditions (1), (2), and $G$ is locally quasiconvex. For $v\in V(\mathcal Z)$, let $H$ be a compactly generated subgroup of the vertex group $G_v$. Let $\Gamma_H,\Gamma_{G_v},$ and $\Gamma_G$ be Cayley-Abels graphs of $H,G_v,$ and $G$, respectively. Let $i:\Gamma_H\to \Gamma_{G_v}$ and $j:\Gamma_{G_v}\to \Gamma_G$ are embeddings. By Lemma \ref{lemma:proper-embedding}, $j$ is a proper embedding. Since $G$ is locally quasiconvex, $H$ is quasiconvex in $G$. This implies that $j\circ i$ is a qi embedding (Lemma \ref{lemma:properties-of-qc}(3)). Then, by \cite[Lemma 2.1(1)]{sardar-tomar}, $i$ is a qi embedding. Hence, by Lemma \ref{lemma:properties-of-qc}(2), $H$ is quasiconvex in $G_v.$

From Theorem \ref{theorem:combination-of-locally-qc}, we have the following.
 
\begin{corollary}\label{corollary:locally-qc-compact-edge-groups}
     Let $(\mathcal G,\mathcal Z)$ be a finite graph of groups such that the vertex groups are hyperbolic TDLC and the edge groups are compact. Then, the fundamental group of $(\mathcal G,\mathcal Z)$ is locally quasiconvex if and only if the vertex groups are locally quasiconvex in $G$. \qed
\end{corollary}
An immediate consequence of Corollary \ref{corollary:locally-qc-compact-edge-groups} is the following, which demonstrates that hyperbolic TDLC groups that are quasiisometric to trees are locally quasiconvex.

\begin{corollary}\label{corollary:quasiisometric-to-tree-locally-qc}
    Suppose $G$ is a hyperbolic TDLC group such that its Cayley-Abels graph is quasiisometric to a locally finite tree. Then, $G$ is locally quasiconvex.   
\end{corollary}
\begin{proof}
    Since $G$ is quasiisometric to a locally finite tree, $G$ splits as a finite graph of compact groups by \cite[Theorem C]{mathieu-dennis-locally-compact-convergence}. Then, $G$ satisfies all the hypotheses of Theorem \ref{theorem:combination-of-locally-qc}. Hence, $G$ is a locally quasiconvex hyperbolic TDLC group.
\end{proof}
\begin{remark}[Compact extension of locally quasiconvex TDLC groups]\label{remark-locally-qc-comapct-extension}
    Let $G$ be a compactly generated TDLC group and $N$ be a compact open normal subgroup of $G$ such that $G/N$ is locally quasiconvex (note that $G/N$ is a discrete group). By Remark \ref{remark:compact-extension}, $G/N$ is finitely generated and $G/N$ has a Cayley graph which is $\Gamma_{G}$ itself. Now, let $V$ be an open subgroup of $G$ and $V=\langle K \rangle$, $K$ is compact in $V$. Then $\{vN:v \in V\}$ is a finitely generated subgroup of $G/N$, since the natural projection map $G \to G/N$ is surjective and continuous, and $G/N$ is discrete. As $G/N$ is locally quasiconvex, $\{vN:v \in V\}$ is quasiconvex in $\Gamma_{G}$. Hence, $G$ is locally quasiconvex.
\end{remark}

\begin{example}\label{example-local-qc}
    (1) Let $d\geq 2$ and let $Aut(T_d)$ denote the automorphism group of the regular tree $T_d$. Then, Cayley-Abels graphs of $Aut(T_d)$ are quasiisometric to $T_d$. Hence, by Corollary \ref{corollary:quasiisometric-to-tree-locally-qc}, $Aut(T_d)$ is locally quasiconvex.

    (2) Let $\mathbb Q_p$ be a group of $p$-adic numbers. Then, it is well-known that ${\rm SL}(2,\mathbb Q_p)$ splits as an amalgamated free product of compact open subgroups \cite{serre-trees}. Hence, by Corollary \ref{corollary:quasiisometric-to-tree-locally-qc}, $\rm{SL}(2,\mathbb Q_p)$ is locally quasiconvex.

    (3) Let $G$ be as in Example \ref{example:normalizer-centralizer}. Then, by \cite{combination-theorem}, $G$ is a hyperbolic TDLC group whose Cayley-Abels graph is quasisometric to a locally finite tree. Then, by Corollary \ref{corollary:quasiisometric-to-tree-locally-qc}, $G$ is a locally quasiconvex group.

    (4) Let $G$ and $M$ be as in Example \ref{example:malnormal-example}. Assume further that $D$ is locally quasiconvex. Then, by Remark \ref{remark-locally-qc-comapct-extension}, $G$ is locally quasiconvex. Then, using Theorem \ref{theorem:combination-of-locally-qc}, we see that $G\ast_M\overline{G}$ is locally quasiconvex.
\end{example}
Given two finitely generated subgroups of a hyperbolic group, it is natural to ask when they will generate a quasiconvex subgroup. In this direction, using the geometry of the Cayley graph of the given hyperbolic group, Gitik \cite{gitik-ping-pong} gave a sufficient condition under which the subgroup generated by two quasiconvex subgroups is isomorphic to an amalgamated free product, and is again quasiconvex.  
Let $\Gamma_G=\Gamma_G(K,U,A)$ be the Cayley-Abels graph of $G$ and $d_A$ be the word metric on $V(\Gamma_G).$ For $g\in G$, by the {\em word length} of $g$, we mean $d_A(U,gU)$. Using the geometry of a Cayley-Abels graph of a compactly generated group, we have an exact analog of Gitik's result in the TDLC setup. The proof of the following theorem follows from the proof \cite[Theorem 1]{gitik-ping-pong}. Hence, we skip its proof.

\begin{theorem}\label{theorem:gitik-qc-combination}.
For $\delta\geq 0,D\geq 0$, let $G$ be a $\delta$-hyperbolic TDLC group, and $A,B\le G$ be open $D$-quasiconvex subgroups.
Let $G_0:=A\cap B$.
Then, there exists a constant $C_0=C_0(G,\delta,D)$ such that the following hold:

\begin{enumerate}
\item
If $A_1\leq A$ and $B_1\leq B$ are open subgroups with $A_1\cap B_1=G_0$, and all nontrivial elements of $A_1$ and $B_1$ of word-length $<C_0$ lie in $G_0$, then 
$\langle A_1,B_1\rangle\;\cong\; A_1 *_{G_0}B_1.$
\item
If $A_1$ and $B_1$ are quasiconvex in $G$, then $\langle A_1,B_1\rangle$ is quasiconvex in $G$. \qed
\end{enumerate}    
\end{theorem}

%%%%%%%%%%%%%%%%%%%%%%%%%%%%%%%%%%%%%%%%%%%%%%%%%%%%%%%%%%%%%%%%%%%%%%%%%%%%%%%%%%%%
\section{Existence of quasiisometric sections for extensions of hyperbolic TDLC groups}\label{section:existence-of-qisection}
In \cite{mosher-hypextns}, Mosher introduced the notion of a quasiisometric section for a surjective homomorphism $p: G \to H$ between finitely generated groups $G$ and $H$. A subset $\Sigma \subseteq G$ mapping onto $H$ is said to be a {\em quasiisometric section} or {\em qi section} if, there exists $k\geq 1,\epsilon\geq 0$ such that, for any $g,g' \in \Sigma$, we have $$\frac{1}{k}d_{H}(pg, pg')- \epsilon \leq d_{G}(g,g') \leq kd_{H}(pg, pg')+ \epsilon,$$ where $d_{G}$ and $d_{H}$ are word metrics with respect to some finite generating sets on $G$ and $H$, respectively. Moreover, for extensions of hyperbolic groups, Mosher proved the following.

\begin{lemma}\cite[Quasi-isometric section lemma]{mosher-hypextns}\label{qi section lemma discrete}
    Given a non-elementary hyperbolic group $H$ and a short exact sequence of finitely generated groups $1 \rightarrow H \rightarrow G \xrightarrow{p} Q \rightarrow 1$, the map $p$ has a quasiisometric section $\Sigma$. In fact, choosing a finite symmetric generating set $S$ for $G$ and letting $p(S)$ be the generating set for $H$, then for all $g,g' \in \Sigma$, $d_{H}(pg,pg')- \epsilon \leq d_{G}(g,g') \leq kd_{H}(pg,pg')+ \epsilon$ for some constants $k \geq 1, \epsilon \geq 0$.
\end{lemma}

 This section aims to extend the previous lemma in the setting of hyperbolic TDLC groups. Since we work in the Cayley-Abels graphs of compactly generated TDLC groups, we first explain the meaning of a quasiisometric section for a surjective homomorphism between two such groups. Let $G$ and $Q$ be compactly generated TDLC groups and let $p:G\to Q$ be an open surjective homomorphism. As in Subsection \ref{subsection:2.2}, for a compact generating set $K$ of $G$ and a compact open subgroup $U$ of $G$, there is a finite symmetric set $A \subset G$ containing identity such that $K \subset AU, G= \langle A\rangle U$ and $AU=UAU$. Then, $p(K)$ is a compact generating set for $Q$, and it is easy to see that $p(K) \subseteq p(A)p(U), Q= \langle p(A)\rangle p(U)$, and $p(A)p(U)=p(U)p(A)p(U)$. Let $\Gamma_G(K,U,A)$ and $\Gamma_Q(p(K),p(U),p(A))$ be the Cayley-Abels graphs of $G$ and $Q$, respectively.

  \begin{definition}
     %Let $G$ and $Q$ be compactly generated TDLC groups and $P:G \longrightarrow Q$ be an open continuous surjective homomorphism. Let $\Gamma_{G}(K,U,A)$ and\\ $\Gamma_{Q}(P(K),P(U),P(A))$ be the Cayley-Abels graphs of $G$ and $Q$, respectively, as constructed as in Subsection \ref{subsection:2.2}. 
     A {\em quasiisometric section} of $p$ is a subset $\Sigma \subseteq V(\Gamma_{G})$ such that $\{p(g)p(U): gU \in \Sigma\}= V(\Gamma_{Q})$, and there exists $k \geq 1$ and $\epsilon \geq 0$ such that,  for any $gU, g'U \in \Sigma$, we have: $$\frac{1}{k}d_{p(A)} (p(g)p(U), p(g')p(U))-\epsilon \leq d_A(gU, g'U) \leq k d_{p(A)} (p(g)p(U),p(g')p(U))+ \epsilon.$$
 \end{definition}

By abuse of notation, we still denote the map from $\Sigma\to V(\Gamma_Q)$ by $p$, and $p(gU)=p(g)p(U)$ for $gU\in \Sigma$. Let $H$ be a normal hyperbolic TDLC subgroup of a hyperbolic TDLC group $G$. Let $\Gamma_H$ and $\Gamma_G$ be the Cayley-Abels graphs of $H$ and $G$, respectively. By Lemma \ref{lemma:conjugation-qi}, for each $g\in G$, $\phi_g$ is a quasiisometry, and thus induces an action by homeomorphism on the Gromov boundary $\partial \Gamma_H$ of $\Gamma_H$, see \cite[Theorem 3.9, III.H]{bridson-haefliger}. Also, there is a natural left translation action of $H$ on $\Gamma_H$. For each $h\in H$, we denote this action by $t_h$ that gives an isometry of $\Gamma_H$. Thus, $t_h$ gives rise to a homeomorphism of $\partial \Gamma_H$. We record the following lemma, which follows straightforwardly. Hence, we skip its proof.

\begin{lemma}\label{lemma:same-induced-homeomorphism}
   The homeomorphisms induced by $\phi_g$ and $t_h$ on $\partial \Gamma_H$ are the same.           \qed
   \end{lemma}

We denote by $\mathcal A_g$ the homeomorphism of $\partial \Gamma_h$ induced by $\phi_g$. Now, we are ready to state an analogous result of Lemma \ref{qi section lemma discrete} in the TDLC setting.

 \begin{proposition}\label{lemma4}
     Given a non-elementary hyperbolic TDLC group $H$ and a short exact sequence of open maps of compactly generated TDLC groups $1 \rightarrow H \rightarrow G \xrightarrow{p} Q \rightarrow 1$, the map $p$ has a quasiisometric section $\Sigma$. In fact, choosing a compact generating set $K$ for $G$ and letting $p(K)$ be the compact generating set for $Q$ we have, for all $gU, g'U \in \Sigma$, $d_{p(A)} (p(g)p(U),p(g')p(U)) \leq d_A(gU, g'U)) \leq k d_{p(A)} (p(g)p(U),p(g')p(U)) + \epsilon_{0}$ for some constants $k \geq 1$ and $\epsilon_{0} \geq 0$.
 \end{proposition}
 Let $\Gamma_H$ be the Cayley-Abels graph of $H$ and let $\partial \Gamma_H$ be the Gromov boundary of $\Gamma_H$. Let $\theta(\partial \Gamma_H)$ be the set of unordered distinct triples. By Proposition \ref{prop:uniform-convergence-action}, $H$ acts on $\theta(\partial\Gamma_H)$ as a uniform convergence group. Let $C$ be a compact subset of $\theta(\partial\Gamma_H)$ such that $HC=\theta(\partial\Gamma_H)$. As discussed above, $G$ also acts on $\partial \Gamma_H$ by homeomorphisms, hence it acts on $\theta(\partial\Gamma_H)$. Then, $\theta(\partial\Gamma_H)=HUC$. Fix an element $\xi\in\theta(\partial\Gamma_H).$

{\bf Definition of $\Sigma$:} $\Sigma:=\{gU\in V(\Gamma_G): \xi\in \mathcal A_g(UC)\}.$

Suppose $gU=g'U$. Then, $g'=gu$ for some $u \in U$. Hence, $\mathcal{A}_{g'}(UC)= \mathcal{A}_{gu}(UC)= \mathcal{A}_{g} (UC)$. Thus, the definition of $\Sigma$ does not depend on the coset representatives of $U$.
 
{\em Proof of Proposition \ref{lemma4}.} 
First of all, we see that $p(\Sigma)=V(\Gamma_Q)$. Note that since $C$ is a fundamental domain, then $\{\mathcal{A}_{h} (UC): h \in H\}$ covers $\theta(\partial\Gamma_H)$. Let $qp(U)\in V(\Gamma_Q)$. Then, there exists $g \in G$ such that $p(g)=q$. Since $\mathcal{A}_{g}$ is a homeomorphism of $\partial \Gamma_{H}$ and $\mathcal{A}_{g}\mathcal{A}_{h}=\mathcal{A}_{gh}$, the sets $\{\mathcal{A}_{gh} (UC): h \in H\}$ also covers $\theta(\partial\Gamma_H)$, so $\xi \in \mathcal{A}_{gh} (UC)$ for some $h$. Thus, there is some representative $gh$ of the coset $gH$ such that $ghU$ is in $\Sigma$. This implies that $p(ghU)=qp(U)$, and hence $p(\Sigma) = V(\Gamma_{Q})$. Clearly, by the choice of the generating set for $G$ and $Q$, $d_{p(A)}(p(g)p(U),p(g')p(U)) \leq d_{A}(gU, g'U)$ for any $gU,g'U \in V(\Gamma_{G})$. Now, let $gU, g’U \in \Sigma$, which implies that $\mathcal{A}_{g} (UC) \cap \mathcal{A}_{g'} (UC) \neq \phi$ as both contain the element $\xi$. This further implies that
$(UC) \cap \mathcal{A}_{g^{-1}g'} (UC) \neq \phi$. To prove the other inequality, in order to show that $\Sigma$ is a quasiisometric section, it suffices to assume $d_{p(A)} (p(g)p(U), p(g')p(U))=1$ and prove that $d_{A}(gU,g'U)$ is bounded above by some uniform constant.

Since, $d_{p(A)}(p(g)p(U),p(g')p(U))=1$, there exists $a\in A$ such that $p(g)p(U)=p(g')p(a)p(U).$ Thus, for some $u\in U$, $p(a^{-1}g'^{-1}gu)=1_Q$. This implies that, for some $h\in H$, $a^{-1}g'^{-1}gu=h$. Since $H$ is normal in $G$, $g'^{-1}g=h'au^{-1}$ for some $h'\in H.$ 
Therefore $UC\cap \mathcal A_{g'^{-1}g}(UC)=UC \cap \mathcal{A}_{h'a} (UC) =UC \cap \mathcal{A}_{h'} (\mathcal{A}_{a} (UC)) \neq \phi$. Since $A$ is finite there are only finitely many choices for $a$; and for each such $a$, since $H$ acts properly on $\theta(\partial\Gamma_H)$, the set $K_{a}:=\{h \in H: UC \cap \mathcal{A}_{h} (\mathcal{A}_{a} (UC)) \neq \phi\}$ is compact and therefore $h'$ belongs $K_a$. Let $c_a$ be the constant as in Lemma \ref{lemma:cosets of cpt set bounded} and let $c=\max\{c_a:a\in A\}$.  Thus, $d_{\Gamma_{G}}(gU,g'U) = d_{\Gamma_{G}} (U, h'aU) \leq d_{\Gamma_{G}} (U, h'U)+ d_{\Gamma_{G}} (h'U, h'aU)\leq c+L$, where $L=\max\{d_A(U,aU):a\in A\}.$  This completes the proof of the proposition.
\qed

\vspace{.2cm}

An immediate consequence of Proposition \ref{lemma4} is the following:
\begin{corollary}
    Suppose $1\to H\to G\to Q\to 1$ is a short exact sequence of open maps of compactly generated TDLC groups such that $H$ is non-elementary hyperbolic. If $G$ is hyperbolic, then $Q$ is also hyperbolic.   \qed
\end{corollary}
\begin{remark}\label{remark:single-valued-qi-section}
    Given a quasiisometric section $\Sigma$, we choose coset representatives of $U$ to get a single valued quasiisometric section $\Sigma '$. Since $p(\Sigma ')=V(\Gamma_{Q})$, there exists $gU \in \Sigma'$ such that $p(g)p(U)= p(U)$. This implies that $p(g) \in p(U)$ and thus $g^{-1}u\in H$ for some $u\in U$. Since $H$ is normal, $g=hu$ for some $h \in H$, and therefore $gU=hU$. Now, using the left translation $t_{h^{-1}}$, one can assume that $\Sigma '$ contains $U$. Note that $t_{h}(\Sigma ')$ is still a single valued quasiisometric section as $t_{h^{-1}}$ preserves the cosets. Therefore, we can assume that $1 \rightarrow H \rightarrow G \rightarrow Q \rightarrow 1$ is an exact sequence of compactly generated TDLC groups with $H, G$ hyperbolic and $\sigma: V(\Gamma_{Q}) \longrightarrow V(\Gamma_{G})$ is a single-valued quasiisometric section containing $U$. 
\end{remark}

%%%%%%%%%%%%%%%%%%%%%%%%%%%%%%%%%%%%%%%%%%%%%%%%%%%%%%%%%%%%%%%%%%%%%%%%%%%%%%%%%%%%%%%%%%

\section{Existence of CT maps for extensions of hyperbolic TDLC groups}\label{section:4}
The goal of this section is to give a proof of Theorem \ref{main theorem}. We closely follow the strategy of Mj (Mitra) \cite{mitra-ct}. Let $H$ and $G$ be as in the statement of Theorem \ref{main theorem}. Let $\Gamma_G$ and $\Gamma_H$ be the Cayley-Abels graphs of $G$ and $H$, respectively and $\iota:\Gamma_H\to\Gamma_G$ be the embedding. Throughout this section, we assume that $\Gamma_H$ is $\delta$-hyperbolic for some $\delta\geq 0$. Firstly, we aim to construct a set $B_{\lambda} \subset \Gamma_{G}$ for a geodesic segment $\lambda \subset \Gamma_{H}$, which turns out to be quasiconvex in $\Gamma_{G}$ (Lemma \ref{lemma:ladder-qc}). Throughout the section, we fix a geodesic segment $\lambda\subset\Gamma_H$ and let $\phi_g$ denote the quasiisometry as in Lemma \ref{lemma:conjugation-qi}. Also, we continue to follow some of the notation from the previous section.
 
 \subsection{Ladder construction} 
 Let $aU_{H},bU_{H}$ denote the endpoints of $\lambda$ and let $\lambda_{g}$ denote a geodesic joining $\phi_{g}(aU_H)$ to $\phi_{g}(bU_H)$ in $\Gamma_{H}$ for $g$ such that $gU \in \sigma(V(\Gamma_{Q}))$. Now, we are ready to define $B_{\lambda}$, which is called the {\em ladder} corresponding to $\lambda$. Define $$B_{\lambda}:=\bigcup_{g:gU \in \sigma(V(\Gamma_{Q}))} t_{g}\iota(\lambda_{g}),$$ where $\sigma$ is a single valued qi section, see Remark \ref{remark:single-valued-qi-section}. As $U \in \sigma(V(\Gamma_{Q}))$, $B_{\lambda}$ contains $\iota(\lambda)$. On $\Gamma_{H}$ define a map $\pi_{g,\lambda}:\Gamma_{H} \longrightarrow \lambda_{g}$ taking $hU_{H}$ to one of the points on $\lambda_{g}$ closest to $hU_{H}$ in the metric $d_{A_{H}}$. Here $\pi_{g,\lambda}$ is defined only on the vertex set, but this is enough for our purposes. Next, we want to define a map $\Pi_{\lambda}:\Gamma_G\to B_{\lambda}$. Let $g'U\in V(\Gamma_G)$. Since $\sigma$ is a single-valued quasiisometric section, there exists a unique $gU \in \sigma(V(\Gamma_{Q}))$ such that $p(g)p(U)=p(g')p(U)$. This implies that $g'^{-1}gu\in H$ for some $u\in U$, and, using the normality of $H$, we see that there exists $h\in H$ such that $g'=ghu$. Therefore, $g'U=ghU=t_g\iota(hU_H).$ Now, define $\Pi_{\lambda}: \Gamma_{G} \rightarrow B_{\lambda}$ by $$\Pi_{\lambda}(g'U)=\Pi_{\lambda}(t_{g}\iota(hU_{H})):= t_{g}\iota(\pi_{g,\lambda}(hU_{H})).$$

 We show that $\Pi_{\lambda}$ does not increase distances by more than a bounded factor. For that, we show that $\pi_{g,\lambda}$ does not increase distances much. 
  This follows from the following lemma whose proof is the same as that of Lemma 3.2 in \cite{mitra-ct}.

 \begin{lemma}\label{lemma:pi_lambda is lipschitz}
     Let $\mu \subset \Gamma_{H}$ be a geodesic segment and let $\pi_{\mu}:\Gamma_{H} \longrightarrow \mu$ be a map taking $hU_{H}$ to one of the points on $\mu$ nearest to $hU_{H}$. Then $d_{A_{H}}(\pi_{\mu}(xU_{H}), \pi_{\mu}(yU_{H})) \leq C_{1}d_{A_{H}}(xU_{H},yU_{H})$ for all $x,y \in H$ where $C_{1}$ depends only on $\delta$.    \qed
 \end{lemma}
 \begin{lemma}\label{finite ball condition}
    Let $1 \rightarrow H \rightarrow G \xrightarrow{p}Q\rightarrow 1$ be as in the statement of Theorem \ref{main theorem} and $\sigma: V(\Gamma_{Q}) \to V(\Gamma_{G})$ be a single valued quasiisometric section obtained from Proposition \ref{lemma4}. Further let $k$ and $\epsilon_{0}$ be as in Proposition \ref{lemma4}. Then $d_{A}(\sigma (p(x)p(U)), \sigma (p(y)p(U))) \leq k_{1}:=k+\epsilon_{0}$ when $d_{A}(xU,yU)=1$.
\end{lemma}
\begin{proof}
    From Proposition \ref{lemma4}, if $d_{A}(xU,yU)=1$ then  $d_{p(A)}(p(x)p(U), p(y)p(U)) \leq 1$. Hence $d_{A} (\sigma(p(x)p(U)), \sigma(p(y)p(U))) \leq k+\epsilon_{0}=k_{1}$.
\end{proof}

\begin{remark}\label{phi_g is a (K,epsilon qi} 
    Since $\Gamma_{G}$ is locally finite, any ball around $U$ in $\Gamma_{G}$ contains finitely many vertices. Therefore, for all $g \in G$ with $d_{\Gamma_{G}}(U,gU) \leq k_{1}$, there exists $K \geq 1$ and $\epsilon \geq 0$ independent of $g$ such that $\phi_{g}$ is a $(K, \epsilon)$-quasiisometry. From now on, we use these constants as quasiisometry constants for $\phi_g$.
\end{remark} 
  The next lemma is important to us for proving Theorem \ref{main theorem} and its proof is the same as the proof of Lemma 3.6 of \cite{mitra-ct}. Thus, we skip its proof.

\begin{lemma}\label{lemma:relation btwn nearest point projections}
     Let $\mu_{1}$ be a geodesic segment in $\Gamma_{H}$ joining $aU_{H}$ to $bU_{H}$ and $\mu_{2}$ be a geodesic segment joining $\phi_{g}(aU_H)$ to $\phi_{g}(bU_H)$ for some $g$ such that $d_{A}(U,gU)\leq k_{1}$. For $sU_{H}\in V(\Gamma_H)$, let $qU_{H}$ be a vertex on $\mu_{1}$ such that $d_{A_{H}}(sU_H,qU_H) \leq d_{A_{H}}(sU_H,xU_H)$ for all $xU_{H} \in \mu_{1}$. Let $rU_{H}$ be a vertex on $\mu_{2}$ such that $d_{A_{H}}(\phi_{g}(sU_H),rU_H) \leq d_{A_H}(\phi_{g}(sU_H),xU_H)$ for $xU_{H} \in \mu_{2}$. Then,  $d_{A_{H}}(\phi_{g}(qU_H),rU_H) \leq C_{2}$ for some $C_{2}$ depending on $K,k,\epsilon,\delta$.    \qed
\end{lemma}

The following theorem shows that $\Pi_{\lambda}$ does not increase distances by more than a bounded factor.

\begin{theorem}
    There exists a constant $C \geq 1$ depending on $k, \epsilon_{0}, K, \epsilon, \delta$ such that for all geodesic segments $\lambda \subset \Gamma_{H}$ and $xU,yU \in V(\Gamma_{G})$, $d_{A}(\Pi_{\lambda}(xU),\Pi_{\lambda}(yU)) \leq Cd_{A}(xU,yU)$.
\end{theorem}
\begin{proof}
     It suffices (by repeated use of the triangle inequality) to prove the theorem when $d_{A}(xU,yU)=1$. We consider the following two cases.
     
{\bf Case 1.} Assume $xU,yU \in t_{g}\iota(\Gamma_{H})$ for some $g$ such that $gU \in \sigma(V(\Gamma_{Q} ))$. Let $xU= t_{g}\iota(x_{1}U_{H}), yU= t_{g}\iota(y_{1}U_{H})$ for some $x_{1}, y_{1} \in H$. Then, $d_{A_{H}}(x_{1}U_{H}, y_{1}U_{H}) = 1.$ Therefore, 
 \begin{align*}
     d_{A}(\Pi_{\lambda}(xU),\Pi_{\lambda}(yU)) &= d_{A}(t_{g}\iota(\pi_{g,\lambda}(x_{1}U_{H})),t_{g}\iota(\pi_{g,\lambda}(y_{1}U_{H})))\\ &= d_{A}(\iota(\pi_{g,\lambda}(x_{1}U_{H})),\iota(\pi_{g,\lambda}(y_{1}U_{H})))\\ &\leq d_{A_{H}}(\pi_{g,\lambda}(x_{1}U_{H}),\pi_{g,\lambda}(y_{1}U_{H}))\\ &\leq C_{1} d_{A_{H}}(x_{1}U_{H},y_{1}U_{H}) \text{ (Lemma \ref{lemma:pi_lambda is lipschitz})}\\ &= C_{1}.
 \end{align*}

{\bf Case 2.} Assume that $d_{A}(xU,yU)=1$ but $xU$ and $yU$ do not belong to $t_g\iota(\Gamma_H)$ for some $g\in G$. Let $xU \in t_{g_{0}}\iota(\Gamma_{H}), yU \in t_{g'}\iota(\Gamma_{H})$ for some $g_{0}\neq g'$ such that $g_{0}U, g'U \in \sigma(V(\Gamma_{Q}))$. Let $xU= t_{g_{0}}\iota(x_{2}U_{H}), yU= t_{g'}\iota(y_{2}U_{H}) $ for some $x_{2}, y_{2} \in H$. Therefore, $\sigma(p(x)p(U))= \sigma(p(g_{0}x_{2}U))= \sigma(p(g_{0}x_{2})p(U))= \sigma(p(g_{0})p(U))= g_{0}U$. Similarly, $\sigma(p(y)p(U))= g'U$. Now, by Lemma \ref{finite ball condition} we have, $d_{A}(\sigma(p(x)p(U)), \sigma(p(y)p(U))) \leq k_{1}$ and thus $d_{A}(g_{0}U,g'U) \leq k_{1}$. Letting $g= g_{0}^{-1}g'$, we see that $yU= t_{g_{0}g}\iota(y_{2}U_{H})$. Let $\pi_{g_{0},\lambda}(x_{2}U_{H})=hU_{H}$ for some $h \in H$. Then, $\Pi_{\lambda}(xU)= \Pi_{\lambda}(g_{0}\iota(x_{2}U_{H}))= g_{0}\iota(\pi_{g_{0},\lambda}(x_{2}U_{H}))=g_{0}hU.$ Now, $t_{g_{0}g}\iota\phi_{g}\pi_{g_{0},\lambda}(x_{2}U_{H})= t_{g_{0}g}\iota(g^{-1}hgU_{H})=g_{0}hgU$ and $t_{g_{0}g}\iota\pi_{g_{0}g,\lambda}.\phi_{g}(x_{2}U_{H})= t_{g_{0}g}\iota\pi_{g_{0}g,\lambda}(g^{-1}x_{2}gU_{H})=\Pi_{\lambda}(t_{g_{0}g}\iota(g^{-1}x_{2}gU_{H}))=\Pi_{\lambda}(g_{0}x_{2}gU).$ Note that $g_{0}x_{2}gU=t_{g_{0}g}\iota(g^{-1}x_{2}gU_{H})= t_{g_{0}g}\iota\phi_{g}(x_{2}U_{H})$. This implies that $g_{0}x_{2}gU \in t_{g_{0}g}\iota(\Gamma_{H})$. Also $yU \in t_{g_{0}g}\iota(\Gamma_{H}).$ Therefore 
  \begin{align*}
      d_{A}(\iota(g^{-1}x_{2}gU_{H}),\iota(y_{2}U_{H})) &= d_{A}(g_{0}x_{2}gU, yU) \\ &\leq d_{A}(g_{0}x_{2}gU, g_{0}x_{2}U)+ d_{A}(g_{0}x_{2}U, yU)\\ &= d_{A}(gU,U)+ d_{A}(xU, yU)\\ &\leq k_{1}+1.
  \end{align*}
Since $\iota$ is a proper embedding, there is a constant $L$ independent of $g, \lambda$ such that $d_{A_{H}}(\phi_{g}(x_{2}U_{H}),y_{2}U_{H}) \leq L.$ Thus, by Lemma \ref{lemma:pi_lambda is lipschitz}, we have that $d_{A_{H}}(\pi_{g_{0}g,\lambda}(\phi_{g}(x_{2}U_{H})),\\ \pi_{g_{0}g, \lambda}(y_{2}U_{H})) \leq C_{1}L.$ This implies that, $d_{A}(\iota(\pi_{g_{0}g, \lambda}(\phi_{g}(x_{2}U_{H}))), \iota(\pi_{g_{0}g, \lambda}(y_{2}U_{H})))\leq C_{1}L$, and therefore $d_{A}(t_{g_{0}g}\iota(\pi_{g_{0}g, \lambda}(\phi_{g}(x_{2}U_{H}))), t_{g_{0}g}\iota(\pi_{g_{0}g, \lambda}(y_{2}U_{H})))\leq C_{1}L,$ that is $d_{A}(\Pi_{\lambda}(yU), \Pi_{\lambda}(g_{0}x_{2}gU)) \leq C_{1}L.$ Since $d_{A}(U,gU) \leq k_{1}$, by Remark \ref{phi_g is a (K,epsilon qi}, $\phi_{g}$ is a $(K,\epsilon)$-quasiisometry. Then $\phi_{g}(\lambda_{g_{0}})$ is a $(K, \epsilon)$-quasigeodesic joining the endpoints of $\lambda_{g_{0}g}$. Therefore, by Lemma \ref{lemma:relation btwn nearest point projections}, $d_{A_{H}}(\phi_{g}(\pi_{g_{0}, \lambda}(x_{2}U_{H})), \pi_{g_{0}g, \lambda}(\phi_{g}((x_{2}U_{H})) \leq C_{2}.$ Therefore, $d_{A}(t_{g_{0}g, \lambda}\iota\phi_{g}(\pi_{g_{0}, \lambda}(x_{2}U_{H})), t_{g_{0}g, \lambda}\iota\pi_{g_{0}g, \lambda}(\phi_{g}((x_{2}U_{H})) \leq C_{2}.$ This implies that $d_{A}(g_{0}hgU, \Pi_{\lambda}(g_{0}x_{2}gU)) \leq C_{2}.$
Now, 
  \begin{align*}
      d_{A}(\Pi_{\lambda}(xU),\Pi_{\lambda}(yU)) &= d_{A}(g_{0}hU,\Pi_{\lambda}(yU))\\ &\leq d_{A}(g_{0}hU, g_{0}hgU)+ d_{A}(g_{0}hgU, \Pi_{\lambda}(g_{0}x_{2}gU)) \\           &+d_{A}(\Pi_{\lambda}(g_{0}x_{2}gU), \Pi_{\lambda}(yU))\\ &\leq k_{1}+ C_{2}+ C_{1}L. 
  \end{align*}
  Taking $C= max \{C_{1}, k_{1}+ C_{2}+ C_{1}L\},$ we obtain, for all $xU,yU$ with $d_{A}(xU,yU)= 1$,  $d_{A}(\Pi_{\lambda}(xU),\Pi_{\lambda}(yU)) \leq C.$
\end{proof}

 \begin{lemma}\label{lemma:ladder-qc}
     There exists $C'=C'(\delta, k, \epsilon_{0}, K, \epsilon)$ such that for all geodesic segments $\lambda \subset \Gamma_{H}$, $B_{\lambda}$ is $C'$-quasiconvex in $\Gamma_{G}$.
 \end{lemma}
{\em Sketch of the proof:} From the proof of \cite[Proposition 4.6(4)]{sardar-krishna}, it follows that, with induced length metric from $\Gamma_G$, the $C$-neighborhood $N_C(B_{\lambda})$ of $B_{\lambda}$ is connected and qi embedded in $\Gamma_G$. Since $\Gamma_G$ is hyperbolic, $N_C(B_{\lambda})$ is quasiconvex in $\Gamma_G$. Hence, $B_{\lambda}$ is $C'$-quasiconvex for some $C'=C'(\delta,k,\epsilon_0,K,\epsilon).$  \qed

\vspace{.2cm}

 Thus, there exists $C'$ independent of $\lambda$ such that every geodesic with endpoints in $B_{\lambda}$ lies in a $C'$-neighborhood of $B_{\lambda}$. In particular, any geodesic joining the endpoints of $i(\lambda)$ lies in a $C'$-neighborhood of $B_{\lambda}$. Thus, $B_{\lambda}$'s are $C'$-uniformly quasiconvex. Next, we record the last one lemma before proving Theorem \ref{main theorem}.

 \begin{lemma}\label{lemma6}
      For all $gU \in \sigma(V(\Gamma_{Q}))$ and $xU \in t_{g}\iota(\lambda_{g})$, there exists $yU \in \iota(\lambda)$ such that $d_{A}(xU,yU) \leq Rd_{p(A)}(p(U), p(x)p(U))$ for some $R\geq 1$ independent of $\lambda$.
 \end{lemma}
 \begin{proof}
     Let $\mu$ be a geodesic segment joining $p(U)$ and $p(x)p(U)$ in $\Gamma_{Q}$. Order the vertices on $\mu$ so that we have a finite sequence $p(U)=y_{0}p(U), y_{1}p(U), \cdots, y_{n}p(U)=p(x)p(U)=p(g)p(U)$ such that $d_{p(A)}(y_{i}p(U), y_{i+1}p(U))=1$ and $d_{p(A)}(p(U),p(x)p(U))=n$. Since $\sigma$ is a quasiisometric section, this gives a sequence $\sigma(y_{i}p(U))=g_{i}U$ such that $d_{A}(g_{i}U,g_{i+1}U) \leq k_{1}$ for $0\leq i\leq (n-1)$.  
     This implies that $d_{A}(U,s_{i}U) \leq k_{1}$, where $s_{i}=g_{i+1}^{-1}g_{i}$. Hence $\phi_{s_{i}}$ is a $(K,\epsilon)$- quasiisometry by Remark \ref{phi_g is a (K,epsilon qi}, where $K,\epsilon$ are independent of $i$. Let $zU \in t_{g_{i+1}}\iota(\lambda_{g_{i+1}})$. Then $zU = t_{g_{i+1}}\iota(qU_{H})$, for some $qU_{H} \in \lambda_{g_{i+1}}$. Now, $t_{g_{i}}\iota\phi_{s_{i}}(qU_{H})= t_{g_{i}}(g_{i}^{-1}g_{i+1}qg_{i+1}^{-1}g_{i}U)= g_{i+1}qg_{i+1}^{-1}g_{i}U$. Since $\phi_{s_{i}}$ is a $(K,\epsilon)$- quasiisometry, $\phi_{s_{i}}(\lambda_{g_{i+1}})$ is a $(K,\epsilon)$- quasigeodesic in $\Gamma_{H}$ joining the endpoints of $\lambda_{g_{i}}$. Then, there exists some $vU_{H} \in \lambda_{g_{i}}$ such that $d_{A_{H}}(vU_{H}, \phi_{s_{i}}(qU_{H})) \leq K'$, where $K'$ depends on $K, \epsilon, \delta$. Therefore, $d_{A}(t_{g_{i}}\iota(vU_{H}), t_{g_{i}}\iota(\phi_{s_{i}}(qU_{H}))) \leq K'$. Let $wU= t_{g_{i}}\iota(vU_{H})$ for some $w \in G$. Therefore, $d_{A}(wU, g_{i+1}qg_{i+1}^{-1}g_{i}U) \leq K'$. Hence, $d_{A}(wU,zU) \leq d_{A}(wU, g_{i+1}qg_{i+1}^{-1}g_{i}U)+ d_{A}(g_{i+1}qg_{i+1}^{-1}g_{i}U, g_{i+1}qU) \leq K'+k_{1}= R$ (say). Therefore, given $zU \in t_{g_{i+1}}\iota(\lambda_{g_{i+1}})$, there is $wU \in t_{g_{i}}\iota(\lambda_{g_{i}})$ such that $d_{A}(wU,zU) \leq R$, $R$ is independent of $\lambda$.

Note that, $g_{n}U=gU$ and $g_0U=U$. Suppose $xU=x_{n}U$. By the above paragraph, there exists $x_{i}U \in t_{g_{i}}\iota(\lambda_{g_{i}})$ for $i=0,1, \cdots, (n-1)$ such that $d_{A}(x_{i}U, x_{i+1}U) \leq R$. Choosing $yU=x_{0}U$ we have $yU \in \iota(\lambda)$ and $d_{A}(xU,yU) \leq Rd_{p(A)}(p(U), p(x)p(U))$.
\end{proof}
     
\subsection{Existence of the CT map for $(H,G)$} With the help of Mitra's criterion, we are now ready to prove Theorem \ref{main theorem}.
\vspace{.2cm}

\textbf{Proof of Theorem \ref{main theorem}.} By Lemma \ref{lemma:mitra-criterion}, it suffices to show that if $\lambda$ is a geodesic segment lying outside an $N$-ball around $U_{H}$ in $\Gamma_{H}$, then any geodesic segment joining the endpoints of $\iota(\lambda)$ in $\Gamma_{G}$ lies outside an $M(N)$-ball around $U$ in $\Gamma_{G}$ and $M(N) \to \infty$ as $N \to \infty$. Since $\iota:\Gamma_{H}\to\Gamma_G$ is $f$-properly embedded for some $f:[0,\infty)\to [0,\infty)$, there exists $f(N)$ such that $\iota(\lambda)$ lies outside the $f(N)$-ball in $\Gamma_{G}$ and $f(N) \to \infty$ as $N \to \infty$. Let $zU$ be a point on a geodesic that joins the endpoints of $\iota(\lambda)$. Since $B_{\lambda}$ is quasiconvex in $\Gamma_G$ containing $\iota(\lambda)$, there is some $xU \in B_{\lambda}$ such that $d_{\Gamma_{G}}(xU,zU) \leq C'$. By Lemma \ref{lemma6}, there is some $yU \in \iota(\lambda)$ such that $d_{A}(xU,yU) \leq Rd_{p(A)}(p(U),p(x)p(U))$. This implies that, using triangle inequality, $d_{A}(U,xU) \geq d_{A}(U,yU)- Rd_{p(A)}(p(U),p(x)p(U)) \geq f(N)- Rd_{p(A)}(p(U),p(x)p(U))$. Also, by the choices of generating sets for $G$ and $Q$, $d_{A}(U,xU) \geq d_{p(A)}(p(U),p(x)p(U))$. By combining these inequalities, $$d_{A}(U,xU) \geq max\{f(N)- Rd_{p(A)}(p(U),p(x)p(U)),d_{p(A)}(p(U),p(x)p(U))\} \geq \frac{f(N)}{R+1}.$$ 

Thus,  $d_{A}(U,zU) \geq d_{A}(U,xU)-  d_{A}(xU,zU) \geq \frac{f(N)}{R+1}-C'= M(N)$ (say). Since $f(N) \to \infty$ as $N \to\infty$ so does $M(N)$.       \qed

\vspace{.2cm}

We end the paper with the following example.

\begin{example}\label{example:ct-map-example}
    Let $D$ be an infinite discrete hyperbolic group that contains an infinite index infinite normal hyperbolic subgroup (for example, let $\Sigma_g$ be a closed connected orientable surface of genus $g\geq 2$ and let $\psi$ be a pseudo-Anosov homeomorphism of $\Sigma_g$. Then, the fundamental group of the mapping torus of $\Sigma_g$ by $\psi$ is an example of such a discrete hyperbolic group), and let $C$ be a compact non-discrete TDLC group. Then, the direct product $G=D\times C$, with the product topology, is a non-discrete TDLC group. By Remark \ref{remark:compact-extension}, $G$ is hyperbolic and the Gromov boundary of $G$ is homeomorphic to the Gromov boundary of $D$. Let $N_D$ be an infinite index infinite normal hyperbolic subgroup of $D$ and let $N_C$ be a compact open normal subgroup of $C$. Then, $N:=N_D\times N_C$ is an infinite index open normal subgroup of $G$, and, by Remark \ref{remark:compact-extension}, $N$ is hyperbolic. Then, by Theorem \ref{main theorem}, the CT map exists for the pair $(N,G)$. 
\end{example}

%%%%%%%%%%%%%%%%%%%%%%%%%%%%%%%%%%%%%%%%%%%%%%%%%%%%%%%%%%%%%%%%%%%%%%%%%%%%%%%%%%

\vspace{.2cm}
\noindent
{\bf Acknowledgements.} 
The authors would like to thank Pierre-Emmanuel Caprace and Yves Cornulier for their helpful comments on the draft version of the paper.
R. Tomar gratefully acknowledges the postdoctoral fellowship from BICMR, Peking University, and the warm hospitality received during his visit to IIT Roorkee.

\vspace{.2cm}
\noindent
{\bf Conflict of interest:} On behalf of all authors, the corresponding author states
that there is no conflict of interest.  
 
\bibliography{ref}
\bibliographystyle{amsalpha}
\end{document}